\documentclass[12pt]{article}
\usepackage[ngerman,english]{babel}
\usepackage[latin1]{inputenc}
\usepackage{csquotes}
\usepackage[normalem]{ulem}
\usepackage{amssymb,amsmath,amsthm, amsfonts}
\usepackage{graphicx}
\usepackage{epsfig}
\usepackage{tikz}
\usepackage[colorlinks]{hyperref}
\usepackage{caption,subcaption}
\usepackage{color}
\usepackage{dsfont}
\usepackage{empheq}
\usepackage{cancel}
\usepackage[a4paper, left=2.0cm, right=2.0cm, top=2.5cm,bottom=2cm]{geometry}

\def\pd#1#2{\frac{\partial#1}{\partial#2}}

\newcommand{\cb}{\color{blue}}

\theoremstyle{plain}
\newtheorem{theorem}{Theorem}[section]
\newtheorem{lemma}{Lemma}[section]
\newtheorem{proposition}{Proposition}[section]

\theoremstyle{definition}

\newtheorem{remark}{Remark}[section]

\newtheorem*{assumption}{Assumptions}

\def\D{{\cal D}}

\def\Om {\Omega }
\def\R {{\mathbb R}}
\def\eps {\varepsilon }
\def\Ker {\text {Ker }}

\numberwithin{equation}{section}

\begin{document}
\allowdisplaybreaks
\title{\large\bf Global boundedness, hair trigger effect, and pattern formation driven by the
parametrization of a nonlocal Fisher-KPP problem}
\author{
{\rm Jing Li\footnote{E-mail address: matlj@163.com}}\\
{\it\small  College of Science, Minzu University of China,
  Beijing, 100081, P.R. China}\\
{\rm Li Chen\footnote{E-mail address: chen@math.uni-mannheim.de}}\\
{\it\small Lehrstuhl f\"ur Mathematik IV, Universit\"at Mannheim, 68131, Mannheim, Germany }\\
{\rm Christina Surulescu\footnote{ E-mail address: surulescu@mathematik.uni-kl.de}}\\
{\it\small Felix-Klein-Zentrum f\"ur Mathematik, Paul-Ehrlich-Stra\ss e 31, 67663 Kaiserslautern, Germany }
}

\date{}

\maketitle

\begin{abstract}
The global boundedness and the hair trigger effect of solutions for the nonlinear nonlocal reaction-diffusion equation
\begin{align*}
\pd ut=\Delta u+\mu u^\alpha(1-\kappa J*u^\beta),\quad\hbox{in} \;\mathbb R^N\times(0,\infty),\; N\geq 1
\end{align*}
with $\alpha\geq1$, $\beta,\mu,\kappa>0$ and $u(x,0)=u_0(x)$ are investigated. Under appropriate assumptions on $J$, it is proved that for any  nonnegative and bounded initial condition, if $\alpha\in[1,\alpha^*)$ with
$\alpha^*=1+\beta$ for $N=1,2$ and $\alpha^*=1+\frac{2\beta}{N}$ for $N>2$, then
the problem has a global bounded classical solution.  Under further assumptions on the initial datum, the solutions satisfying $0\leq u(x,t)\leq\kappa^{-\frac1\beta}$ for any $(x,t)\in\mathbb R^N\times[0,+\infty)$ are shown to converge to $\kappa^{-\frac1\beta}$ uniformly on any compact subset of $\mathbb R^N$, which is known as the hair trigger effect. 1D numerical simulations of the above nonlocal reaction-diffusion equation are performed and the effect of several combinations of parameters and convolution kernels on the solution behavior is investigated.
The results motivate a discussion about some conjectures arising from this model and further issues to be studied in this context. A formal deduction of the model from a mesoscopic formulation is provided as well.
\end{abstract}

\vskip5mm

{\small

}
\vskip5mm

\section{Introduction}

In this work we study the nonlinear nonlocal reaction-diffusion equation\footnote{Here we consider the equation to be already nondimensionalized.}
\begin{align}
&\pd ut=\Delta u+\mu u^\alpha(1-\kappa J*u^\beta),&&(x,t)\in \mathbb R^N\times(0,\infty),
\label{eq:1-1}
\\
&u(x,0)=u_0(x),&&x\in\mathbb R^N,\label{eq:1-2}
\end{align}
where $\alpha\geq1$, $\beta,\mu,\kappa>0$, $N\geq1$, $J(x)$ is a competition
kernel with
\begin{align}
0\leq J\in L^1(\mathbb R^N),\quad\int_{\mathbb R^N} J(x)dx=1,\quad \hbox{and}\; \inf_{B(0,\delta_0)}J>\eta \quad\hbox{for some}\; \delta_0>0, \eta>0,\label{eq:1-3}
\end{align}
where $B(0,\delta_0)=\{x\in\mathbb R^N:|x_i|\leq\delta_0,i=1,2,\cdots,N\}$ and
$$J*u^\beta(x,t)=\int_{\mathbb R^N}J(x-y)u^\beta(y,t)dy.$$

This problem can be seen to characterize the evolution of a population of density $u$, whose individuals are moving by diffusion and interaction. Their interaction modus determines the faith of the population with respect to growth or decay: the reaction term describes the joint influence of a nonlinear growth accounting for a weak Allee effect and of concurrence for available resources (prevention of overcrowding). The latter takes a nonlocal form; several individuals interact in a space/phenotypic trait/etc. domain, thereby sampling all occupancy information therein.
Such problems arise e.g., when modeling emergence and evolution of a biological
species \cite{4,ref9,ref10,ref22,ref32,1}. Thereby the respective population is structured by a phenotypical trait and its individuals infer two essential interactions: mutation and selection. In this context $u(x,t)$ represents the density of a population having phenotype $x$ at time $t$. The mutation process, which acts as a diffusion operator on the trait space, is modeled by a classical diffusion operator, whereas the selection process is described by the nonlocal term $u^\alpha(1-J*u^\beta)$. Similar nonlocal reaction terms also
occur in describing natural selection of cancer cells, which leads to the emergence of therapy-resistent clones \cite{20,21}.

Equation \eqref{eq:1-1} is a particular case of a more general monospecies setting which can be deduced in various ways. The perhaps simplest one (see e.g. \cite{britton89}) starts from the local reaction-diffusion equation
\begin{equation*}
u_t=D\Delta u+r(u)u
\end{equation*}
and lets the growth rate $r(u)$ depend not only on the population density at a certain location $x$, but also at the other points in some domain of interest (which can be the whole space):
\begin{equation*}
r(u)=ru^{\alpha -1}(1-J*u^\beta ).
\end{equation*}

Indeed, a spatially heterogeneous population can exceed locally its carrying capacity, which in the usual Fisher equation (logistic growth) would simply lead to decay. By letting, however, the population use resources/signals available at more or less proximal sites such decay is not necessary to happen. It is known, for instance,  that cells in a tissue are able to communicate with each other by way of thin protrusions (lamellipodia, filopodia, cytonemes, nanotubes) which can reach at long distances with respect to the cell size, see \cite{GM17,SSM17} and references therein. Moreover, clustering together or organizing in groups may even provide advantages, depending on the competition strength; this can apply to cells \cite{biofilms}, but also to animals cooperating for hunt, associating in schools or swarms, or simply undergoing sexual reproduction (case $\alpha =2$).

Another way to obtain a PDE of the type \eqref{eq:1-1}, with or without diffusion, is by relying on individual-based formulations involving stochastic processes and performing some appropriate upscaling, see e.g. \cite{CFM06}; we also refer to \cite{Ichikawa} for an instance of deducing a reaction-diffusion system with nonlocality in the reaction terms by using master equations and mean field limits.

Yet another approach \cite{BBNS12} uses kinetic transport equations to derive by a hydrodynamic limit PDEs for which an equation of the form
\eqref{eq:1-1} is a particular case. We shortly illustrate in the Appendix its concrete application. 


When the interaction kernel $J$ in \eqref{eq:1-1} is replaced by the Dirac delta function the
so called generalized Fisher-KPP equation is obtained as a local reaction-diffusion equation. In \cite{ref99} and \cite{ref18}, the Fisher-KPP equation
$$\frac{\partial u}{\partial t}=\Delta u+u(1-u)$$
was introduced to model the spreading of some advantageous gene in a population. It is well known that any solution $u(x,t)$ with a nonnegative and nontrivial initial data, tends to $1$ as $t\to\infty$, locally uniformly in $x\in\mathbb R^N$. This is referred to as the \textit{hair trigger effect} \cite{ref3}.
When accounting for a weak Allee effect the above equation becomes
$$\frac{\partial u}{\partial t}=\Delta u+u^\alpha(1-u)$$
for $\alpha>1$. As stated in \cite{Lou}, an immediate difficulty arises when trying to apply standard comparison methods, since the equilibrium $u=0$ is degenerate. It turns out that the dynamics of solutions is much more complicated and interesting than that for $\alpha=1$. In \cite{ref3}, Aronson and Weinberger showed that for $N\geq2$, the hair trigger effect remains valid as long as $1\leq\alpha\leq1+\frac{2}{N}$, whereas some small initial data may lead to extinction, or quenching, when $\alpha>1+\frac{2}{N}$.

For $J(x)=1$, which corresponds to the situation of blind competition,
with general $\alpha\geq1$ and $\beta\geq1$, the problem has been studied
in \cite{BCL,BC} in terms of the existence of solutions both in bounded
and unbounded domains, respectively. Moreover, from the analysis of \cite{Covi}, denote $\lambda_1$ the first eigenvalue of the operator $\Delta+\mu$ with Neumann boundary condition and let $\phi_1$ be a positive eigenfunction associated with $\lambda_1$, for any $\beta\geq1$, the positive solution of
$$\frac{\partial u}{\partial t}=\Delta u+u\left(\mu-\int_{\Omega}|u^\beta(t,y)|dy\right)$$
in bounded domain converges uniformly to $\left(\frac{-\lambda_1}{\int_\Omega |\phi_1^\beta(y)|dy}\right)^{\frac1\beta}$ if $\lambda_1<0$, while for
$\lambda_1\geq0$, there is no positive stationary solution and
$u(x,t)\to0$ as $t\to\infty$.

For $J$ satisfying \eqref{eq:1-3}, the consumption of resources at the space/phenotypic trait point $x$ depends on individuals located in some area around this point. As stated in \cite{ref2}, introducing  nonlocal intraspecific competition for resources changes the properties of solutions of this equation. Some progress has recently been attained in this direction for the so called nonlocal Fisher-KPP equation
\begin{align}
\frac{\partial u}{\partial t}=\Delta u+\mu u(1-J*u),\label{eq:1-4}
 \end{align}
for which $u=1$ is a stationary solution. The latter is stable in the case of the local equation, but it can lose its stability for the nonlocal one. If it becomes unstable, then a periodic in space stationary solution bifurcates from it \cite{ref13,ref19,ref21}. This phenomenon is observed in the study of travelling wave solutions. If the Fourier transform of $J$ is everywhere positive or if $\mu$ is small enough, then it is known that travelling waves necessarily connect $0$ to $1$ (see \cite{7,4,6,15,5}), while if $\mu$ is large, then $u=1$ can indeed become unstable and Turing patterns appear \cite{5,99}. Similar results were obtained for the bistable case
\begin{align}
\frac{\partial u}{\partial t}=\Delta u+\mu u^2(1-J*u)-du,
\end{align}
in our previous work \cite{LC}, where $-du$ is the mortality term and $d$ is the death rate. For the long time behavior of solutions, in a recent work \cite{Pou}, it was proved that for the nonlocal Fisher-KPP equation \eqref{eq:1-4} and under the assumption
$$\forall f\in L^2(\Omega),\quad\iint_{\Omega\times\Omega}J(x-y)f(x)f(y)dxdy\geq0,$$ the solution converges to $1$ uniformly in a bounded domain $\Omega$. Moreover, \cite{Covi} considered
 \begin{align}
\frac{\partial u}{\partial t}=\Delta u+u\left(\mu-\kappa \int_{\Omega}(J_0(y)+\varepsilon J(x,y))|u^\beta(y)|dy\right)
\end{align}
on a bounded domain with Neumann boundary condition with $\beta=1,2$. Upon relying on non-linear relative entropy identities and an orthogonal decomposition, for any strictly positive and smooth $J_0(y)$ and $J(x,y)$, denote $\lambda_1$ the first eigenvalue of the operator $\Delta+\mu$ with Neumann boundary condition, it was proved that for $\lambda_1<0$, there exists $\varepsilon^*>0$, if $\varepsilon\in(0,\varepsilon^*)$, there exists a unique positive steady state which is asymptotically stable. While for $\lambda_1\geq0$, there is no positive stationary solution. The stability of the steady state for the problem with no diffusion was investigated in \cite{PE}.

 For unbounded domains, however, whether the hair trigger effect will occur or not for such nonlocal problems is an issue far from obvious, as also mentioned in  \cite{Covi,Pou}.

In this paper, depending on the balance between the weak Allee effect and the overcrowding  avoidance effect, we find sufficient conditions for the global boundedness  of solutions for \eqref{eq:1-1} and the hair trigger effect in long time behavior, 
where the localization 
technique in the proofs is motivated by \cite{AA}. As a byproduct of the global boundedness, we give a quasi-maximum principle and find that the bounds of solutions are influenced (in a certain way) by the parameters $\alpha$, $\beta$, $\mu$, $\kappa$, the kernel $J$, the spatial dimension $N$ and the $L^\infty$ norm of the initial data. The main results of this paper are the following.

\begin{theorem}\label{Th:1.1}
Suppose $\alpha\in[1,\alpha^*)$  with
\begin{align*}
\alpha^*=\left\{\begin{array}{ll}
1+\beta,\quad N=1,2,
\\
1+\frac{2\beta}{N},\quad N>2
\end{array}
\right.
\end{align*}
where $\beta , \mu, \kappa>0$ and \eqref{eq:1-3} holds. Then for every initial data $0\leq u_0\in L^\infty(\mathbb R^N)$, the nonnegative solution of \eqref{eq:1-1}--\eqref{eq:1-2} exists and is globally bounded in time, that is, there exists $M>0$ such that
\begin{align}
0\leq u(x,t)\leq M,\qquad \forall\ (x,t)\in\mathbb R^N\times[0,+\infty).\label{00}
\end{align}

Moreover, for any $K>1$, there exists $\mu^*>0$ such that for $\mu\in(0,\mu^*)$, \eqref{00} holds with
$$M=K\max\left\{1,\left(\frac{A}{\kappa}\right)^{\frac{s^*-2}{s^*(\beta+1-\alpha)-2\beta}},\|u_0\|_{L^\infty(\mathbb R^N)}\right\},$$
where
\begin{align*}
s^*=\left\{\begin{array}{ll}
&+\infty,\quad N=1,2,
\\
&\frac{2N}{N-2},\quad N>2,
\end{array}\right.
\end{align*}
and
\begin{align*}
A=\frac{4\left[4\sqrt{2}\max\left\{1,\delta_0 G(s^*,N)\right\}\right]^{\frac {s^*}{s^*-1}}}{\delta_0^N\eta}
\end{align*}
with $\delta_0$ and $\eta$ introduced in \eqref{eq:1-3} and $G(s^*,N)$ the constant that appears in Sobolev's inequality.
\end{theorem}

For the hair trigger effect, we impose further restrictions on the initial data $u_0(x)$:

$(A)$\quad $0\leq u_0(x)\leq\kappa^{-\frac1\beta}.$

$(B)$\quad For some $\delta>0$, $\int_{B(x,\delta)}\ln u_0(s)ds\in L^\infty(\mathbb R^N)$ holds for $\alpha=1$ and
$\int_{B(x,\delta)}u_0^{1-\alpha}(s)ds\in L^\infty(\mathbb R^N)$
holds for $\alpha>1$.

\begin{theorem}\label{Th:1.2}
Suppose $(A)$ and $(B)$ are satisfied for $u_0$ and \eqref{eq:1-3} holds for $J$. If $u(x,t)$ is a global solution of \eqref{eq:1-1}--\eqref{eq:1-2} with $0\leq u(x,t)\leq\kappa^{-\frac1\beta}$ for any $(x,t)\in\mathbb R^N\times[0,+\infty)$, then
$$\lim_{t\to\infty}u(x,t)=\kappa^{-\frac1\beta}$$
locally uniformly in $\mathbb R^N$.
\end{theorem}

\begin{remark}
In Theorem \ref{Th:1.2}, the hair trigger effect is proved for those solutions staying between the two constant solutions: $0$ and $\kappa^{-\frac1\beta}$. In Theorem \ref{Th:1.1}, by a series of careful calculations, we give an explicit upper bound which is larger than $\kappa^{-\frac1\beta}$.
Without comparison principle, finding the conditions on the coefficients and the initial data to ensure that $0\leq u(x,t)\leq\kappa^{-\frac1\beta}$ is still to be addressed. Though it is not proved, such solutions do exist. At least, the constant solutions are of this kind. Besides, for the case $\alpha=\beta=1$, the monotone travelling solutions connecting $0$ and $\kappa^{-\frac1\beta}$, which are proved in literatures for small $\mu's$, are also between $0$ and $\kappa^{-\frac1\beta}$.
\end{remark}

The structure of the paper is the following.
Sections \ref{sec:global-exist} and \ref{sec:hair-trigger} are devoted to prove Theorem{\cb s} \ref{Th:1.1} and \ref{Th:1.2}, respectively. Numerical simulations together with a discussion of the results are presented in Section \ref{sec:num-discussion}.
A formal derivation of the model from a mesoscopic setting in the framework of kinetic transport equations is provided in the attachment.

\section{Global boundedness of solutions}\label{sec:global-exist}

The global existence of a solution will be obtained by a local wellposedness result and a uniform in time $L^\infty$ estimate.

The local existence, uniqueness, and nonnegativity of solutions for this parabolic problem can be readily obtained by the estimates of heat kernel and the fixed point theorem applied in a closed bounded set in $L^\infty((0,T)\times\mathbb R^N)$ for suitably small $T>0$, followed by regularity arguments. This approach has been extensively employed in order to deal with the inhomogeneous heat equation, and is an adaption  of that used in  \cite{fujita}. For convenience, we list the corresponding result below.

\begin{proposition}\label{Prop-2.1}
Assume the initial data $0\leq u_0(x)\in L^\infty(\mathbb R^N)$. Then there is a maximal existence time $T_{max}\in(0,\infty]$ and $u\in C([0,T_{max}),L^\infty(\mathbb R^N))\cap C^{2,1}(\mathbb R^N\times(0,T_{max}))$ such that $u$ is the unique non-negative classical solution of problem \eqref{eq:1-1}--\eqref{eq:1-2}.
Furthermore, if $T_{max}<+\infty$, then
$$\lim_{t\to T_{max}}\|u(\cdot,t)\|_{L^\infty(\mathbb R^N)}=\infty.$$
\end{proposition}

The following result is an application of Ghidaglia's lemma (see \cite{3}, Lemma 5.1). It is a generalized version of Lemma 4.1 in \cite{2}. For completeness, we also give the proof.
\begin{lemma}\label{Lemma1}
Assume $y_k(t)\geq0$, $k=0,1,2,\cdots$ are $C^1$ functions for $t>0$,  satisfying
\begin{align}
y_k'(t)+c_ky_k(t)\leq c_kA_k\max\{1,\sup_{t\geq0}y^2_{k-1}(t)\},\label{4.6}
\end{align}
where $c_k>0$, $A_k=\overline{a}2^{Dk}\geq1$ with $\overline{a}$, $D$, being  positive constants. Assume also that there exists a constant $K>0$ such that $y_k(0)\leq K^{2^k}$. Then for all $m\geq1$, it holds
$$y_k(t)\leq(2\overline{a})^{2^{k-m+1}-1}2^{D\left(2(2^{k-m}-1)+m2^{k+m-1}-k\right)}
\max\{\sup_{t\geq0}y_{m-1}^{2^{k-m+1}}(t),K^{2^{k}},1\}.$$
\end{lemma}

\begin{proof} From \eqref{4.6} we obtain
\begin{align*}
(e^{c_kt}y_k(t))'\leq c_kA_ke^{c_kt}\max\{1,\sup_{t\geq0}y_{k-1}^{2}(t)\}
\end{align*}
and then
\begin{align*}
y_k(t)&\leq (1-e^{-c_kt})A_k\max\{1,\sup_{t\geq0}y_{k-1}^{2}(t)\}+e^{-c_kt}y_k(0)
\\
&\leq2A_k\max\{1,\sup_{t\geq0}y_{k-1}^{2}(t),y_k(0)\}
\\
&\leq2A_k\max\{\sup_{t\geq0}y_{k-1}^{2}(t),K^{2^k},1\}.
\end{align*}
By an iterative procedure we obtain
\begin{align*}
y_k(t)&\leq2A_k(2A_{k-1})^2(2A_{k-2})^{2^2}(2A_{k-3})^{2^3}\cdots(2A_m)^{2^{k-m}}\max\{\sup_{t\geq0}y_{m-1}^{2^{k-m+1}}(t),K^{2^{k}},1\}
\\
=&(2\overline{a})^{\sum_{i=0}^{i=k-m}2^{i}}2^{D\sum_{i=0}^{i=k-m}(k-i)2^{i}}\max\{\sup_{t\geq0}y_{m-1}^{2^{k-m+1}}(t),K^{2^{k}},1\}
\\
=&(2\overline{a})^{2^{k-m+1}-1}2^{D\left(2(2^{k-m}-1)+m2^{k+m-1}-k\right)}
\max\{\sup_{t\geq0}y_{m-1}^{2^{k-m+1}}(t),K^{2^{k}},1\}.
\end{align*}
\end{proof}

Now we are ready to proceed with the global $L^\infty$ estimates.

\noindent
{\bf Proof of Theorem \ref{Th:1.1}}. For any $x\triangleq(x_{1},x_{2},\cdots,x_{N})\in\mathbb R^N$, choose $0<\delta\leq\frac12\delta_0$, and denote $B(x,\delta):=\{y\triangleq(y_1,y_2,
\cdots,y_n)\in\mathbb R^N\ |\ |y_i-x_i|\leq\delta,i=1,2,\cdots,N\}$ with $|B(x,\delta)|=(2\delta)^N$. Multiply \eqref{eq:1-1} by $pu^{p-1}\varphi_\varepsilon$ with $p\ge 1$, $\varphi_\varepsilon\in C_0^\infty(B(x,\delta))$, and $\varphi_\varepsilon(y)\to1$ locally uniformly in $B(x,\delta)$ as $\varepsilon\to0$. Integrating by parts over $B(x,\delta)$ we obtain
\begin{align}
&\frac{\partial}{\partial t}\int_{B(x,\delta)}u^p\varphi_\varepsilon dy+\frac{4(p-1)}{p}
\int_{B(x,\delta)}|\nabla u^{\frac p2}|^2\varphi_\varepsilon dy\nonumber
\\
=&\int_{B(x,\delta)}\Delta u^p\varphi_\varepsilon dy
+p\mu\int_{B(x,\delta)}u^{p+\alpha-1}(1-\kappa J*u^\beta)\varphi_\varepsilon dy.\label{3.0}
\end{align}
Taking $\varepsilon\to0$, we obtain
\begin{align}
&\frac{\partial}{\partial t}\int_{B(x,\delta)}u^pdy+\frac{4(p-1)}{p}
\int_{B(x,\delta)}|\nabla u^{\frac p2}|^2dy\nonumber
\\
=&\int_{B(x,\delta)}\Delta u^pdy
+p\mu\int_{B(x,\delta)}u^{p+\alpha-1}(1-\kappa J*u^\beta)dy\notag
\\
\leq&\Delta\int_{B(x,\delta)}u^pdy+p\mu\int_{B(x,\delta)}u^{p+\alpha-1}dy
\left(1- \eta\kappa \int_{B(x,\delta)}u^\beta dy\right),\label{eq:3.1}
\end{align}
where we have used the fact that
for any $y\in B(x,\delta)$, $J(z-y)\geq\eta$ for $z\in B(y,2\delta)$ and $B(x,\delta)\subset B(y,2\delta)$, then
$$J*u^\beta(y,t)\geq\eta\int_{B(y,2\delta)}u^\beta(z,t)dz
\geq\eta\int_{B(x,\delta)}u^\beta(z,t)dz$$
and
$$\int_{B(x,\delta)}\Delta u^pdy=\int_{B(0,\delta)}\Delta u^p(y+x,t)dy=\Delta\int_{B(0,\delta)}u^p(y+x,t)dy=\Delta\int_{B(x,\delta)}u^pdy.$$
Now we proceed to estimate the term
$\int_{B(x,\delta)}u^{p+\alpha-1}dy$.
The proof is divided into three steps, namely: the $L^p$ estimates, the $L^\infty$ estimates, and a quasi-maximum principle.

\textbf{Step 1. $L^p$ estimates.} Firstly by H\"older's inequality and the Sobolev embedding theorem, for
\begin{align}
p\geq\max\{\beta+1-\alpha,1\},\label{p}
\end{align}
\begin{align}
\max\{\frac{N(\alpha-1)}{p},\frac{2(\alpha-1)}{p},1\}<r\leq\frac{2(p+\alpha-1)}{p},\label{rr}
\end{align}
and
\begin{align}\label{s}
s=\left\{\begin{array}{ll}
&+\infty,\quad N=1,
\\
&\frac{2N}{N-2},\quad N>2,
\end{array}\right.
\qquad
\frac{2pr}{pr-2(\alpha-1)}<s<+\infty,\quad  N=2,
\end{align}
we get
\begin{align}
&\int_{B(x,\delta)}u^{p+\alpha-1}dy=\|u^{\frac p2}\|_{L^{\frac{2(p+\alpha-1)}{p}}(B(x,\delta))}^{\frac{2(p+\alpha-1)}{p}}\leq \|u^{\frac p2}\|_{L^{s}(B(x,\delta))}^{\frac{2\lambda(p+\alpha-1)}{p}}\|u^{\frac p2}\|_{L^r(B(x,\delta))}^{\frac{2(1-\lambda)(p+\alpha-1)}{p}}
\label{3.2}
\\
\leq &\left(S(N,\delta)\|u^{\frac p2}\|_{W^{1,2}(B(x,\delta))}\right )^{\frac{2\lambda(p+\alpha-1)}{p}}\|u^{\frac p2}\|_{L^r(B(x,\delta))}^{\frac{2(1-\lambda)(p+\alpha-1)}{p}}\notag
\\
\leq&\left(2S(N,\delta)\|\nabla u^{\frac p2}\|_{L^2(B(x,\delta))}\right)^{\frac{2\lambda(p+\alpha-1)}{p}}\|u^{\frac p2}\|_{L^r(B(x,\delta))}^{\frac{2(1-\lambda)(p+\alpha-1)}{p}}\notag
\\
\quad&+\left(2S(N,\delta)\|u^{\frac p2}\|_{L^2(B(x,\delta))}\right)^{\frac{2\lambda(p+\alpha-1)}{p}}
\|u^{\frac p2}\|_{L^r(B(x,\delta))}^{\frac{2(1-\lambda)(p+\alpha-1)}{p}},\nonumber
\end{align}
where
\begin{align*}
S(N,\delta)=\sqrt{2}\max\{(2\delta)^{N(\frac1s-\frac12)},(2\delta)^{1-\frac N2+\frac Ns}G(s,N)\}
\end{align*}
is the Sobolev embedding constant(see Theorem 2.1 and Theorem 3.1 in \cite{Sobolev}) and
\begin{align}
\lambda=\frac{\frac{p}{2(p+\alpha-1)}-\frac{1}{r}}{\frac1s-\frac1r},
\label{3.10}
\end{align}
then it is easy to verify from \eqref{p} and \eqref{rr} that
\begin{align}
\lambda\in[0,1),\quad 2\lambda(p+\alpha-1)/p\in[0,2).\label{3.9}
\end{align}
On the other hand, by Poincar\'e's inequality, we have
\begin{align*}
\|u^{\frac p2}-\overline{u^{\frac p2}}\|_{L^{2}(B(x,\delta))}&\leq P(N,\delta)\|\nabla u^{\frac p2}\|_{L^{2}(B(x,\delta))},
\end{align*}
where $P(N,\delta)=C(N)\delta$ is the Poincar\'e constant (see Theorem 3.3 in \cite{Sobolev}) and $\bar u$ represents the average of $u$ over $B(x,\delta)$. Then
\begin{align}
\|u^{\frac p2}\|_{L^{2}(B(x,\delta))}&\leq\|\overline{u^{\frac p2}}\|_{L^2(B(x,\delta))}+P(N,\delta)\|\nabla u^{\frac p2}\|_{L^{2}(B(x,\delta))}\nonumber
\\
=&\|u^{\frac p2}\|_{L^1(B(x,\delta))}|B(x,\delta)|^{-\frac12}+P(N,\delta)\|\nabla u^{\frac p2}\|_{L^{2}(B(x,\delta))}\label{L2}
\\
\leq&\|u^{\frac p2}\|_{L^r(B(x,\delta)}|B(x,\delta)|^{\frac12-\frac{1}{r}}+P(N,\delta)\|\nabla u^{\frac p2}\|_{L^{2}(B(x,\delta))}.\nonumber
\end{align}
Inserting \eqref{L2} into \eqref{3.2} we obtain
\begin{align}
&\int_{B(x,\delta)}u^{p+\alpha-1}dy\nonumber
\\
\leq&\left(2S(N,\delta)\|\nabla u^{\frac p2}\|_{L^2(B(x,\delta))}\right)^{\frac{2\lambda(p+\alpha-1)}{p}}\|u^{\frac p2}\|_{L^r(B(x,\delta))}^{\frac{2(1-\lambda)(p+\alpha-1)}{p}}\label{3.15}
\\
&+\left[2S(N,\delta)\left(\|u^{\frac p2}\|_{L^r(B(x,\delta))}|B(x,\delta)|^{\frac12-\frac{1}{r}}+P(N,\delta)\|\nabla u^{\frac p2}\|_{L^{2}(B(x,\delta))}\right)\right]^{\frac{2\lambda(p+\alpha-1)}{p}}\|u^{\frac p2}\|_{L^r(B(x,\delta))}^{\frac{2(1-\lambda)(p+\alpha-1)}{p}}\nonumber
\\
&\leq 2\left(C_1(N,\delta)\|\nabla u^{\frac p2}\|_{L^2(B(x,\delta))}\right)^{\frac{2\lambda(p+\alpha-1)}{p}}\|u^{\frac p2}\|_{L^r(B(x,\delta))}^{\frac{2(1-\lambda)(p+\alpha-1)}{p}}+\left(C_2^\lambda(N,\delta,r)\|u^{\frac p2}\|_{L^r(B(x,\delta))}\right)^{\frac{2(p+\alpha-1)}{p}},\nonumber
\end{align}
where
\begin{align}
C_1(N,\delta)=2S(N,\delta)(1+2P(N,\delta)), \quad C_2(N,\delta,r)=4S(N,\delta)(2\delta)^{\frac{N(r-2)}{2r}}.\label{1.11}
\end{align}
Denote
\begin{equation}\label{erstes-Q}
Q:=\frac{2(1-\lambda)(p+\alpha-1)}{p-\lambda(p+\alpha-1)}.
\end{equation}
Then from \eqref{3.9} we have $$Q=2\left(1+\frac{\alpha-1}{p-\lambda(p+\alpha-1)}\right)\geq\frac{2(p+\alpha-1)}{p},$$ which together with Young's inequality leads to
\begin{align}
 &2\left(C_1(N,\delta)\|\nabla u^{\frac p2}\|_{L^2(B(x,\delta))}\right)^{\frac{2\lambda(p+\alpha-1)}{p}}\|u^{\frac p2}\|_{L^r(B(x,\delta))}^{\frac{2(1-\lambda)(p+\alpha-1)}{p}}
+\left(C_2^\lambda(N,\delta,r)\|u^{\frac p2}\|_{L^r(B(x,\delta))}\right)^{\frac{2(p+\alpha-1)}{p}}\nonumber
\\
\leq& \frac{p-1}{\mu p^2}\|\nabla u^{\frac p2}\|_{L^2(B(x,\delta))}^2+C_3(p,\mu,N,\delta)\|u^{\frac p2}\|_{L^r(B(x,\delta))}^Q
+\left(C_2^\lambda(N,\delta,r)\|u^{\frac p2}\|_{L^r(B(x,\delta))}\right)^{\frac{2(p+\alpha-1)}{p}}\nonumber
\\
\leq& \frac{p-1}{\mu p^2}\|\nabla u^{\frac p2}\|_{L^2(B(x,\delta))}^2+\left(C_3(p,\mu,N,\delta)+C_2^{\frac{2\lambda(p+\alpha-1)}{p}}(N,\delta,r)\right)\|u^{\frac p2}\|_{L^r(B(x,\delta))}^Q\label{3.3}
\\
&+C_2^{\frac{2\lambda(p+\alpha-1)}{p}}(N,\delta,r),\nonumber
\end{align}
where $C_3(p,\mu,N,\delta)=2
\left(\frac{2\mu p^2C_1^2(N,\delta)}{p-1}\right)^{\frac{\lambda(p+\alpha-1)}{p-\lambda(p+\alpha-1)}}$.

\noindent
Next we proceed to estimate the term $\|u^{\frac p2}\|_{L^r}^Q$. Notice that for $p\geq\max\{\beta-\alpha+1,1\}$,
\begin{align}
\|u^{\frac p2}\|_{L^r(B(x,\delta))}^Q&\leq\|u^{\frac p2}\|_{L^{\frac{2\beta}{p}}(B(x,\delta))}^{\theta Q}
\|u^{\frac p2}\|_{L^{\frac{2(p+\alpha-1)}{p}}(B(x,\delta))}^{(1-\theta)Q}\nonumber
\\
&= \left(\|u^{\frac p2}\|_{L^{\frac{2\beta}{p}}(B(x,\delta))}^{\frac{2\beta}{p}}
\|u^{\frac p2}\|_{L^{\frac{2(p+\alpha-1)}{p}}(B(x,\delta))}^{\frac{2(p+\alpha-1)}{p}}\right)
^{\frac{Q\theta p}{2\beta}}\|u^{\frac p2}\|_{L^{\frac{2(p+\alpha-1)}{p}}(B(x,\delta))}^{(1-\theta)Q-\frac{(p+\alpha-1)\theta Q}{\beta}},\label{3.4}
\end{align}
where $\theta=\frac{\frac2r-\frac{p}{p+\alpha-1}}{\frac{p}{\beta}-\frac{p}{p+\alpha-1}}$.
Choose
\begin{align}
r=\frac{p+\alpha-1+\beta}{p},\label{r}
 \end{align}
 which obviously satisfies \eqref{rr} and then $\theta=\frac{\beta}{p+\alpha-1+\beta}$. Then we have $$(1-\theta)Q-(p+\alpha-1)\theta Q/\beta=0$$ and
\begin{align}
\|u^{\frac p2}\|_{L^r(B(x,\delta))}^Q\leq \left(\|u^{\frac p2}\|_{L^{\frac{2\beta}{p}}(B(x,\delta))}^{\frac{2\beta}{p}}\|u^{\frac p2}\|_{L^{\frac{2(p+\alpha-1)}{p}}(B(x,\delta))}^{\frac{2(p+\alpha-1)}{p}}\right)
^{\frac{Q\theta p}{2\beta}}.\label{3.44}
\end{align}
By using the definition of $\lambda$ in \eqref{3.10} and the above choices of $r$ and $\theta$,  we have that
\begin{align}
\frac{Q\theta p}{2\beta}=\frac{sp-2(p+\alpha-1)}{s(p-\alpha+1+\beta)-2(p+\alpha-1+\beta)}.\label{1.7}
\end{align}
Under the assumption
\begin{align*}
1\leq\alpha<\alpha^*=\left\{\begin{array}{ll}
1+\beta,\quad N=1,2,
\\
1+\frac{2\beta}{N},\quad N>2,
\end{array}
\right.
\end{align*}
\eqref{1.7} implies $$\frac{Q\theta p}{2\beta}<1.$$ Furthermore, from \eqref{3.44} and Young's inequality, we have
\begin{align}
&(C_3(p,\mu,N,\delta)+C_2^{\frac{2\lambda(p+\alpha-1)}{p}}(N,\delta,r))
\|u^{\frac p2}\|_{L^r(B(x,\delta))}^Q\nonumber
\\
\leq & \frac{\eta\kappa}{4}\|u^{\frac p2}\|_{L^{\frac{2\beta}{p}}(B(x,\delta))}^{\frac{2\beta}{p}}\|u^{\frac p2}\|_{L^{\frac{2(p+\alpha-1)}{p}}(B(x,\delta))}^{\frac{2(p+\alpha-1)}{p}}
+C_4(p,\mu,N,\delta,r,\eta,\kappa),\label{3.5}
\end{align}
where $C_4(p,\mu,N,\delta,r,\eta,\kappa)=\left(C_3(p,\mu,N,\delta)+C_2^{\frac{2\lambda(p+\alpha-1)}
{p}}(N,\delta,r)\right)^{\frac{2\beta}{2\beta-Q\theta p}}\left(\frac{4}{\eta\kappa}\right)^{\frac{Q\theta p}{2\beta-Q\theta p}}$, $\eta $ is the constant in \eqref{eq:1-3}.
Then inserting \eqref{3.3} and \eqref{3.5} into \eqref{3.15},  we obtain
\begin{align}
&2p\mu\int \limits _{B(x,\delta)}u^{p+\alpha-1}dy\label{3.6}
\\
\leq&\frac{2(p-1)}{p}\int \limits _{B(x,\delta)}|\nabla u^{\frac p2}|^2dy+\frac{p\eta\mu\kappa}{2}\int\limits _{B(x,\delta)}u^{p+\alpha-1}dy\int\limits _{B(x,\delta)}u^\beta dy+2p\mu C_5(p,\mu,N,\delta,r,\eta,\kappa)\nonumber
\end{align}
with $C_5(p,\mu,N,\delta,r,\eta,\kappa)=C_4(p,\mu,N,\delta,r,\eta,\kappa)+C_2^{\frac{2\lambda(p+\alpha-1)}{p}}(N,\delta,r)$.
From \eqref{eq:3.1}, \eqref{3.6} and the fact that $$\int_{B(x,\delta)}u^p dy\leq\int_{B(x,\delta)}u^{p+\alpha-1}dy+(2\delta)^N,$$
we obtain
\begin{align}
\frac{\partial}{\partial t}\int_{B(x,\delta)}u^pdy+p\mu\int_{B(x,\delta)}u^pdy\leq
\Delta\int_{B(x,\delta)}u^pdy+\left(2C_5(p,\mu,N,\delta,r,\eta,\kappa)+(2\delta)^N\right)p\mu.\label{y3.7}
\end{align}
\noindent
Next we compare \eqref{y3.7} with the following ordinary differential equation
\begin{align}
\left\{\begin{array}{ll}\label{3.8}
\frac{d}{dt}w+p\mu w=\left(2C_5(p,\mu,N,\delta,r,\eta,\kappa)+(2\delta)^N\right)p\mu,
\\
w(0)=(2\delta)^N\|u_0\|^p_{L^\infty(\mathbb R^N)}.
\end{array}\right.
\end{align}
By the comparison principle, we obtain
\begin{align*}
\int_{B(x,\delta)}u^pdy\leq w&=w(0)e^{-p\mu t}+(2C_5(p,\mu,N,\delta,r,\eta,\kappa)+(2\delta)^N)(1-e^{-p\mu t})
\\
&\leq 2C_5(p,\mu,N,\delta,r,\eta,\kappa)+(2\delta)^N(1+\|u_0\|^p_{L^\infty(\mathbb R^N)})
\end{align*}
for any $(x,t)\in\mathbb R^N\times[0,\infty)$. Then for any $p\geq\max\{\beta-\alpha+1,1\}$, with the explicit representation of the constant $C_5$, we obtain
\begin{align}
\|u\|_{L^p(B(x,\delta))}\leq&\left[2\left(2\left(\frac{2\mu p^2C_1^2(N,\delta)}{p-1}\right)^{\frac{\lambda(p+\alpha-1)}{p-\lambda(p+\alpha-1)}}+C_2^{\frac{2\lambda(p+\alpha-1)}
{p}}(N,\delta,r)\right)^{\frac{2\beta}{2\beta-Q\theta p}}\left(\frac{4}{\eta\kappa}\right)^{\frac{Q\theta p}{2\beta-Q\theta p}}\right.\nonumber
\\
&\left.+2C_2^{\frac{2\lambda(p+\alpha-1)}{p}}(N,\delta,r)+(2\delta)^N(1+\|u_0\|^p_{L^\infty(\mathbb R^N)})\right]^{\frac1p},\label{3.1}
\end{align}
which tends to $+\infty$ as $p\to\infty$.
More precisely, by \eqref{3.10} and \eqref{1.7}, the exponents have the following explicit representations:
\begin{align*}
\frac{\lambda(p+\alpha-1)}{p-\lambda(p+\alpha-1)}=
\frac{p+\alpha-1-\beta}{(p+\alpha-1+\beta)(1-\frac2s)-2(\alpha-1)},
\end{align*}
\begin{align*}
\frac{2\lambda(p+\alpha-1)}{p}=
\frac{s(p+\alpha-1-\beta)}{ps-(p+\alpha-1+\beta)},
\end{align*}
\begin{align*}
\frac{2\beta}{2\beta-Q\theta p}
=\frac{(s-2)(p+\beta)-(s+2)(\alpha-1)}{s(\beta+1-\alpha)-2\beta},
\end{align*}
and
\begin{align*}
\frac{Q\theta p}{2\beta-Q\theta p}
=\frac{sp-2(p+\alpha-1)}{s(\beta+1-\alpha)-2\beta}.
\end{align*}

\textbf{Step 2. $L^\infty$ estimate via Moser iteration.} Now we advance to deduce the $L^\infty$ estimates for $u$ by an iterative procedure. Denote
\begin{align}
p_k:=2^k+h,\label{1.5}
\end{align}
with
\begin{align}
h:=\frac{2(s-1)(\alpha-1)}{s-2}\label{1.10}
\end{align}
chosen to verify \eqref{QR} and $s$ defined in \eqref{s}.
Then the estimate \eqref{3.1} in Step $1$ for $$p=p_{m-1}=2^{m-1}+h\quad(m\geq1)$$ gives the starting point of the iteration. By taking $p=p_k$ in \eqref{eq:3.1}, we have
\begin{align}
&\frac{\partial}{\partial t}\int_{B(x,\delta)}u^{p_k}dy+\frac{4(p_k-1)}{p_k}
\int_{B(x,\delta)}|\nabla u^{\frac{p_k}{2}}|^2dy+p_k\mu\eta\kappa\int_{B(x,\delta)}u^{p_k+\alpha-1}dy\int_{B(x,\delta)}u^\beta dy\nonumber
\\
\leq&\Delta\int_{B(x,\delta)}u^{p_k}dy+p_k\mu\int_{B(x,\delta)}u^{p_k+\alpha-1}dy.\label{eq:3.2}
\end{align}
From \eqref{3.15} and \eqref{3.3}, we obtain
\begin{align}
&\int_{B(x,\delta)}u^{p_k+\alpha-1}dy\nonumber
\\
\leq &\frac{p_k-1}{\mu p_k^2}\|\nabla u^{\frac {p_k}{2}}\|_{L^2(B(x,\delta))}^2+C_2^{\frac{2\lambda_k(p_k+\alpha-1)}{p_k}}(N,\delta,r_k)\label{3.11}
\\
&+\left(2
\left(\frac{2\mu p_k^2C_1^2(N,\delta)}{p_k-1}\right)^{\frac{\lambda_k(p_k+\alpha-1)}{p_k-\lambda_k(p_k+\alpha-1)}}
+C_2^{\frac{2\lambda_k(p_k+\alpha-1)}{p_k}}(N,\delta,r_k)\right)\left(\int_{B(x,\delta)}u^{p_{k-1}}\right)^{\frac{Q_k}{r_k}},\nonumber
\end{align}
where
\begin{align*}
r_k:=\frac{2p_{k-1}}{p_k},
\quad
\lambda_k:=\frac{\frac{p_k}{2(p_k+\alpha-1)}-\frac {p_k}{2p_{k-1}}}{\frac1s-\frac {p_k}{2p_{k-1}}},\quad
Q_k:=\frac{2(1-\lambda_k)(p_k+\alpha-1)}{p_k-\lambda_k(p_k+\alpha-1)}.
\end{align*}
Then by a direct calculation, we obtain
\begin{align}
0<\lambda_k<1,\quad0<\frac{2\lambda_k(p_k+\alpha-1)}{p_k}<2,\nonumber
\end{align}
and
\begin{align}
\frac{Q_k}{r_k}=2.\label{QR}
\end{align}
Notice the fact that
\begin{align}
\int_{B(x,\delta)}u^{p_k} dy\leq\int_{B(x,\delta)}u^{p_k+\alpha-1}dy+(2\delta)^N.\label{3.18}
\end{align}
Substituting \eqref{3.11} and \eqref{3.18} into \eqref{eq:3.2} leads to
\begin{align}
&\frac{\partial}{\partial t}\int_{B(x,\delta)}u^{p_k}dy+p_k\mu\int_{B(x,\delta)}u^{p_k}dy\label{eq:3.4}
\\
\leq&\Delta\int_{B(x,\delta)}u^{p_k}dy
+2p_k\mu C_0
\left(\int_{B(x,\delta)}u^{p_{k-1}}dy\right)^2+2p_k\mu C_0\nonumber
\end{align}
with
\begin{align}
C_0=2
\left(\frac{2\mu p_k^2C_1^2(N,\delta)}{p_k-1}\right)^{\frac{\lambda_k(p_k+\alpha-1)}{p_k-\lambda_k(p_k+\alpha-1)}}
+C_2^{\frac{2\lambda_k(p_k+\alpha-1)}{p_k}}(N,\delta,r_k)+(2\delta)^N.\label{C}
\end{align}
By direct calculations, we have
\begin{align}
\frac{\lambda_k(p_k+\alpha-1)}{p_k-\lambda_k(p_k+\alpha-1)}=\frac{s}{s-2}.\label{1.3}
\end{align}
Further estimates for the constants that appear in $C_0$ are given in the following for $k\geq1$:
\begin{align}
r_k\geq1\qquad\frac{p_k}{p_k-1}\leq2,\qquad\frac{2\lambda_k(p_k+\alpha-1)}{p_k}\leq\alpha+1.\label{1.4}
\end{align}
Using \eqref{1.3} and \eqref{1.4}, from \eqref{C}, we obtain
\begin{align*}
C_0\leq &2\left(4C_1^2(N,\delta)\mu
(2^k+h)\right)^{\frac{s}{s-2}}+C_2^{\alpha+1}(N,\delta,1)+(2\delta)^N
\\
\leq&\left(2\left(4C_1^2(N,\delta)\mu(1+h)\right)^{\frac{s}{s-2}}+C_2^{\alpha+1}(N,\delta,1)+(2\delta)^N\right)2^{\frac{sk}{s-2}}.
\end{align*}
Denote
\begin{align}
a_0:=2\left(4C_1^2(N,\delta)\mu(1+h)\right)^{\frac{s}{s-2}}+C_2^{\alpha+1}(N,\delta,1)+(2\delta)^N,\label{1.12}
\end{align}
then from \eqref{eq:3.4}, we obtain
\begin{align*}
&\frac{\partial}{\partial t}\int_{B(x,\delta)}u^{p_k}dy+p_k\mu\int_{B(x,\delta)}u^{p_k}dy
\\
\leq&\Delta\int_{B(x,\delta)}u^{p_k}dy
+4p_k\mu a_02^{\frac{sk}{s-2}}\max\left\{1,\sup_{t\geq0}
\left(\int_{B(x,\delta)}u^{p_{k-1}}dy\right)^2\right\}.
\end{align*}
For $k\geq m\geq1$, let $y_k(t)$ be the solution of the following iterating ordinary differential equation
\begin{align*}
&y_k'(t)+p_k\mu y_k(t)=4p_k\mu a_02^{\frac{sk}{s-2}}\max\{\sup_{t\geq0}y_{k-1}^2(t),1\},
\\
&y_k(0)=\|u_0\|_{L^\infty(\mathbb R^N)}^{p_k}(2\delta)^N
\end{align*}
with $$y_{m-1}(t)=\sup_{x\in\mathbb R^N}\int_{B(x,\delta)}u^{p_{m-1}}dx$$
 being the starting point of the iteration. Now we are ready to use Lemma \ref{Lemma1}.
For $\overline{a}=\max\{1,4a_0\}$, it is obvious that
$\overline{a}2^{\frac{sk}{s-2}}\geq1$ for all $k\geq m.$
Furthermore, $\delta$ can be chosen such that $(2\delta)^N<\frac{1}{\|u_0\|_{L^\infty(\mathbb R^N)}^h}$, therefore
$$y_k(0)=\|u_0\|_{L^\infty(\mathbb R^N)}^{p_k}(2\delta)^N\leq\|u_0\|_{L^\infty(\mathbb R^N)}^{2^k}.$$
From Lemma \ref{Lemma1} with $c_k=p_k\mu$, $D=\frac{s}{s-2}$, we obtain
\begin{align*}
y_k(t) \leq(2\overline{a})^{2^{k-m+1}-1}2^{\frac{s}{s-2}\left(2(2^{k-m}-1)+m2^{k-m+1}-k\right)}
\max\left\{\sup_{t\geq0}y_{m-1}^{2^{k-m+1}}(t),\|u_0\|_{L^\infty(\mathbb R^N)}^{2^{k}},1\right\}
\end{align*}
and then by the comparison principle, we have
\begin{align*}
&\|u(\cdot,t)\|_{L^{p_k}(B(x,\delta))}
\\
\leq&\left((8a_0)^{2^{k-m+1}-1}2^{\frac{s}{s-2}\left(2(2^{k-m}-1)+m2^{k-m+1}-k\right)}\right)^{\frac{1}{p_k}}
\max\left\{\sup_{t\geq0}y_{m-1}^{\frac{2^{k-m+1}}{p_k}}(t),\|u_0\|_{L^\infty(\mathbb R^N)}^{\frac{2^{k}}{p_k}},1\right\}.
\end{align*}
By letting $k\to\infty$, we obtain
\begin{align*}
&\|u(\cdot,t)\|_{L^\infty(B(x,\delta))}
\\
\leq& (8a_0)^{\frac1{2^{m-1}}}2^{\frac{s(1+m)}{(s-2)2^{m-1}}}\max\left\{\sup_{(x,t)\in\mathbb R^N\times[0,\infty)}\left(\int_{B(x,\delta)}u^{2^{m-1}+h}dy\right)^{\frac{1}{2^{m-1}}}, \|u_0\|_{L^\infty(\mathbb R^N)}, 1\right\}.
\end{align*}
Since $x\in\mathbb R^N$ and $t\in[0,\infty)$ are arbitrary and the boundedness of $$\sup_{(x,t)\in\mathbb R^N\times[0,\infty)}\left(\int_{B(x,\delta)}u^{2^{m-1}+h}dy\right)^{\frac{1}{2^{m-1}}}$$ for any $m\geq1$ is verified in Step $1$,  we have
\begin{align}
\|u\|_{L^\infty(\mathbb R^N\times[0,\infty))}\leq M\label{1.6}
\end{align}
with
\begin{align}
M=(8a_0)^{\frac1{2^{m-1}}}2^{\frac{s(1+m)}{(s-2)2^{m-1}}}
\max\left\{\sup_{(x,t)\in\mathbb R^N\times[0,\infty)}\left(\int_{B(x,\delta)}u^{2^{m-1}+h}dy\right)^{\frac{1}{2^{m-1}}}, \|u_0\|_{L^\infty(\mathbb R^N)}, 1\right\}.\label{M}
\end{align}
Therefore, the global boundedness of $u$ is obtained.

\textbf{Step 3. "Quasi"-maximum principles with small $\mu$'s}

We optimize $M$ introduced in \eqref{M} in the following by using the flexibility of $m$.
Notice that
\begin{align*}
C_2(N,\delta,r)=4\sqrt{2}\max\{(2\delta)^{N(\frac1s-\frac1r)},(2\delta)^{1-\frac Nr+\frac Ns}G(s,N)\}
\end{align*}
with $r$ given in \eqref{r} and $s$ in \eqref{s}. Moreover,
\begin{align*}
\lim_{m\to\infty}C_2\left(N,\delta,\frac{2^{m-1}+h+\alpha-1+\beta}{2^{m-1}+h}\right)
=4\sqrt{2}(2\delta)^{(\frac1s-1)N}\max\left\{1,2\delta G(s,N)\right\}.
\end{align*}
From \eqref{3.1}, if $\mu<\frac1{2C_1^2(N,\delta)(2^{m-1}+h)^2}$, denote $C_2:=C_2\left(N,\delta,\frac{2^{m-1}+h+\alpha-1+\beta}{2^{m-1}+h}\right)$ for simplicity,  we have
\begin{align*}
&\int_{B(x,\delta)}u^{2^{m-1}+h}dy
\\
\leq&2\left(2\left(\frac{1}{2^{m-1}+h-1}\right)^{\frac{2^{m-1}+h-1+\alpha-\beta}{(2^{m-1}+h-1+\alpha+\beta)(1-\frac2s)-2(\alpha-1)}}+
C_2^{\frac{s(2^{m-1}+h-1+\alpha-\beta)}
{(s-1)(2^{m-1}+h)-(\alpha+\beta-1)}}\right)^{\frac{(s-2)(2^{m-1}+h+\beta)-(2+s)(\alpha-1)}{s(\beta+1-\alpha)-2\beta}}
\\
&\cdot\left(\frac{4}{\kappa\eta}\right)^{\frac{(s-2)(2^{m-1}+h)-2(\alpha-1)}{s(\beta+1-\alpha)-2\beta}}
+2C_2^{\frac{s(2^{m-1}+h-1+\alpha-\beta)}
{(s-1)(2^{m-1}+h)-(\alpha+\beta-1)}}+(2\delta)^N(1+\|u_0\|^{2^{m-1}+h}_{L^\infty(\mathbb R^N)}):=H(m).
\end{align*}
Notice that
\begin{align}
&\lim_{m\to\infty}(H(m))^{\frac1{2^{m-1}}}\label{1.2}
\\
=&\max\left\{\left(4\sqrt{2}(2\delta)^{(\frac1s-1)N}\max\left\{1,2\delta G(s,N)\right\}\right)^{\frac{s(s-2)}{(s-1)(s(\beta+1-\alpha)-2\beta)}}
\left(\frac{4}{\kappa\eta}\right)^{\frac{s-2}{s(\beta+1-\alpha)-2\beta}},\|u_0\|_{L^\infty(\mathbb R^N)},1\right\}\nonumber
\end{align}
On the other hand, we obtain
\begin{align}
\lim_{m\to\infty}(8a_0)^{\frac{1}{2^{m-1}}}2^{\frac{s(1+m)}{(s-2)2^{m-1}}}=1.\label{3.19}
\end{align}
According to the above discussion, if we let $m$ go to infinity there will be no positive $\mu$ such that the maximum principle holds. However, we can get the following relaxed version of maximum principle. Namely, for arbitrary $K>1$, from \eqref{1.6}, due to \eqref{1.2} and \eqref{3.19}, there exists a large $m$ (which depends only on $K$) such that  for $\mu\in(0,\mu^*)$ with  $\mu^*=\frac{1}{2C_1^2(N,\delta)(2^{m-1}+h)^2}$, we have
\begin{align}\label{nochne_schaetzung2}
\|u\|_{L^\infty(\mathbb R^N\times[0,\infty))}\leq K\max\left\{1,\left(\frac{A}{\kappa}\right)^{\frac{s^*-2}{s^*(\beta+1-\alpha)-2\beta}},\|u_0\|_{L^\infty(\mathbb R^N)}\right\}
\end{align}
with
\begin{align*}
A=\frac{4\left(4\sqrt{2}\max\left\{1, \delta_0G(s^*,N)\right\}\right)^{\frac {s^*}{s^*-1}}}{\delta_0^N\eta},
\end{align*}
where $\delta$ is chosen as $\frac{\delta_0}{2}$ for small $\delta_0$ and $s$ can be chosen as $s^*$ with
\begin{align*}
s^*=\left\{\begin{array}{ll}
&+\infty,\quad N=1,2,
\\
&\frac{2N}{N-2},\quad N>2.
\end{array}\right.
\end{align*}
Theorem \ref{Th:1.1} is proved.

\section{Long time behavior (hair trigger type effect)}\label{sec:hair-trigger}

Now we consider the long time behavior of the solution of
\eqref{eq:1-1}-\eqref{eq:1-2}.
\begin{proposition}\label{Prop-3.1}
Under the assumptions of Theorem \ref{Th:1.2}, for all $t>0$, the function
$$F(x,t)=\int_{B(x,\delta)}h(u^\beta(y,t))dy$$
is nonnegative and satisfies
\begin{align}
\partial_tF(x,t)\leq \Delta F(x,t)-D(x,t)\label{4.4}
\end{align}
with
\begin{equation*}
h(s)=\left\{
\begin{array}{ll}
&\frac{s}{\beta}-\frac1{\kappa\beta}\ln s-\frac1{\kappa\beta}\left(1+\ln \kappa\right),\quad
\alpha=1,
\\
&\frac{s^{\frac{1+\beta-\alpha}{\beta}}}{1+\beta-\alpha}
-\frac{s^{\frac{1-\alpha}{\beta}}}{\kappa(1-\alpha)}+\kappa^{\frac{\alpha-1-\beta}{\beta}}\left(\frac{1}{1-\alpha}-\frac{1}{1+\beta-\alpha}\right),\quad
\alpha>1
\end{array}\right.
\end{equation*}
and $D(x,t)=\frac12\eta\mu\kappa(2\delta)^N\int_{B(x,\delta)}(\kappa^{-1}-u^\beta(y,t))^2dy$.
\end{proposition}

\begin{proof} Noticing that
\begin{equation*}
h'(s)=\left\{
\begin{array}{ll}
&\frac1\beta-\frac{1}{\kappa\beta s},\quad
\alpha=1,
\\
&\frac1\beta \left(s^{\frac{1-\alpha}{\beta}}-\kappa^{-1} s^{\frac{1-\alpha-\beta}{\beta}}\right),\quad
\alpha>1
\end{array}\right.
\end{equation*}
and $h'(s)<0$ for $0<s<\kappa^{-1}$ and $h'(s)>0$ for $s>\kappa^{-1}$, we obtain that $h(s)\geq h(\kappa^{-1})=0$ and $F(x,t)$ is nonnegative.

\noindent
For the global solution $u$ satisfying $0\leq u(x,t)\leq\left(\frac1{\kappa}\right)^\frac1\beta$ for all $(x,t)\in\mathbb R^N\times[0,\infty)$, the positivity follows from the fact that
\begin{align*}
u(\cdot,t)&=G(\cdot,t)*u_0+\mu\int_0^t\left(u^\alpha(\cdot,s)(1-\kappa J*u^\beta(\cdot,s))
\right)*G(\cdot,t-s)ds
\\
&\geq G(\cdot,t)*u_0>0,
\end{align*}
with $G(x,t)=\frac{e^{-\frac{|x|^2}{4t}}}{(4\pi t)^{\frac N2}}$ the heat kernel.

\noindent
Test \eqref{eq:1-1} by $(u^{\beta-\alpha}-\kappa^{-1} u^{-\alpha})\varphi_\varepsilon$ with $\varphi_\varepsilon\in C_0^\infty(B(x,\delta))$, $\varphi_\varepsilon\to1$ in $B(x,\delta)$ as $\varepsilon\to0$. Integrating by parts over $B(x,\delta)$
we obtain
\begin{align}
&\quad\frac{\partial }{\partial t}\int_{B(x,\delta)}h(u^\beta)\varphi_\varepsilon dy\nonumber
\\
&=\int_{B(x,\delta)}\Delta u(u^{\beta-\alpha}-\kappa^{-1} u^{-\alpha})\varphi_\varepsilon dy+\mu\kappa\int_{B(x,\delta)}(u^\beta-\kappa^{-1})(\kappa^{-1}-J*u^\beta)\varphi_\varepsilon dy\notag
\\
&= \int_{B(x,\delta)}\Delta h(u^\beta)\varphi_\varepsilon dy
 -\frac{4(\beta-\alpha)}{(\beta-\alpha+1)^2}\int_{B(x,\delta)}|\nabla u^{\frac{\beta-\alpha+1}{2}}|^2\varphi_\varepsilon dy-\frac\alpha\kappa\int_{B(x,\delta)}u^{-\alpha-1}|\nabla u|^2\varphi_\varepsilon dy \notag
\\
&\quad+\mu\kappa\int_{B(x,\delta)}\int_{\mathbb R^N}(u^\beta(y,t)-\kappa^{-1})(\kappa^{-1}-u^\beta(z,t))J(z-y)\varphi_\varepsilon(y) dzdy.\label{4.1}
\end{align}
Taking $\varepsilon\to0$, we obtain
\begin{align}
\quad\frac{\partial }{\partial t}F(x,t)&=\Delta F(x,t)
-\frac{4(\beta-\alpha)}{(\beta-\alpha+1)^2}\int_{B(x,\delta)}|\nabla u^{\frac{\beta-\alpha+1}{2}}|^2 dy-\frac\alpha\kappa\int_{B(x,\delta)}u^{-\alpha-1}|\nabla u|^2dy \nonumber
\\
&\quad+\mu\kappa\int_{B(x,\delta)}\int_{\mathbb R^N}(u^\beta(y,t)-\kappa^{-1})(\kappa^{-1}-u^\beta(z,t))J(z-y) dzdy.\label{eq:dt-F}
\end{align}
Then for $\delta<\frac{\delta_0}{2},$ noticing $0<u\leq\kappa^{-\frac1\beta}$, then
\begin{align}
&\quad\int_{B(x,\delta)}\int_{\mathbb R^N}(u^\beta(y,t)-\kappa^{-1})(\kappa^{-1}-u^\beta(z,t))J(z-y)dzdy\notag
\\
&\leq\int_{B(x,\delta)}\int_{B(x,\delta)}(u^\beta(y,t)-\kappa^{-1})(\kappa^{-1}-u^\beta(z,t))J(z-y)dzdy\nonumber
\\
& =-\int_{B(x,\delta)}\int_{B(x,\delta)}(u^\beta(y,t)-\kappa^{-1})^2J(z-y)dzdy\nonumber
\\
&\quad+\int_{B(x,\delta)}
\int_{B(x,\delta)}(u^\beta(y,t)-\kappa^{-1})(u^\beta(y,t)-u^\beta(z,t))J(z-y)dzdy\nonumber
\\
&\leq-\int_{B(x,\delta)}\int_{B(x,\delta)}(u^\beta(y,t)-\kappa^{-1})^2J(z-y)dzdy
+\varepsilon\int_{B(x,\delta)}\int_{B(x,\delta)}(u^\beta(y,t)-\kappa^{-1})^2J(z-y)dzdy\nonumber
\\
&\quad+C(\varepsilon)\int_{B(x,\delta)}\int_{B(x,\delta)}(u^\beta(y,t)-u^\beta(z,t))^2J(z-y)dzdy\nonumber
\\
&\le -(1-\varepsilon)\eta(2\delta)^N\int_{B(x,\delta)}(u^\beta(y,t)-\kappa^{-1})^2dy\nonumber
\\
&\quad
+C(\varepsilon)\int_{B(x,\delta)}\int_{B(x,\delta)}(u^\beta(y,t)-u^\beta(z,t))^2J(z-y)dzdy\label{4.2}
\end{align}
and
\begin{align}
&\quad\int_{B(x,\delta)}\int_{B(x,\delta)}(u^\beta(y,t)-u^\beta(z,t))^2J(z-y)dzdy\label{4.5}
\\
\nonumber&\leq\int_{B(x,\delta)}\int_{B(x,\delta)}\left|\int_0^1\nabla u^\beta(y+\theta(z-y),t)d\theta\right|^2|z-y|^2 J(z-y)dzdy
\\
\nonumber&\leq\int_{B(x,\delta)}\int_{B(x,\delta)}\int_0^1\left|\nabla u^\beta(y+\theta(z-y),t)\right|^2|z-y|^2 J(z-y)d\theta dzdy.
\end{align}
Changing the variables $y'=y+\theta(z-y)$,  $z'=z-y$, then
\begin{equation*}
\left|\begin{array}{cc}
\frac{\partial y'}{\partial y} & \frac{\partial y'}{\partial z}\\
\frac{\partial z'}{\partial y}& \frac{\partial z'}{\partial z}
\end{array}\right|=1.
\end{equation*}
For any  $\theta\in[0,1]$, $y\in B(x,\delta)$,  $z\in B(x,\delta)$, we have $y'\in B((1-\theta)x+\theta z,(1-\theta)\delta)$, $z'\in B(x-y,\delta)$.
Noticing $B((1-\theta)x+\theta z,(1-\theta)\delta)\subseteq B(x,\delta)$ and $B(x-y,\delta)\subseteq B(0,2\delta)$, we obtain
\begin{align}
&\int_{B(x,\delta)}\int_{B(x,\delta)}\int_0^1\left|\nabla u^\beta(y+\theta(z-y),t)\right|^2|z-y|^2 J(z-y)d\theta dzdy\nonumber
\\
\leq&\int_0^1d\theta\int_{B(x,\delta)}\int_{B(0,2\delta)}\left|\nabla u^\beta(y',t)\right|^2|z'|^2 J(z') dz'dy'\nonumber
\\
=&(2\delta)^2\int_{B(x,\delta)}\left|\nabla u^\beta(y',t)\right|^2dy'\label{4.3}
\\
\leq&\frac{4}{\kappa^{2+(\alpha-1)/\beta}}\delta^2\beta^2\int_{B(x,\delta)}u^{-\alpha-1}(y,t)\left|\nabla u(y,t)\right|^2dy.\nonumber
\end{align}
Combing \eqref{4.5}, \eqref{4.3} and \eqref{4.2}, then inserting into \eqref{eq:dt-F}, if $\alpha\leq\beta$, for $0<\delta<\min\{\frac{\delta_0}{2},\sqrt{\frac{\alpha\kappa^{(\alpha-1)/\beta}}{4\mu C(\varepsilon)\beta^2}}\}$, we obtain
\begin{align}
\frac{\partial }{\partial t}F(x,t)
\leq\Delta F(x,t)-(1-\varepsilon)\eta(2\delta)^N\mu\kappa\int_{B(x,\delta)}(\kappa^{-1}-u^\beta(y,t))^2dy.\label{4.8}
\end{align}
For $\alpha>\beta$, noticing that
$$
\frac{4(\alpha-\beta)}{(\beta-\alpha+1)^2}\int_{B(x,\delta)}|\nabla u^{\frac{\beta-\alpha+1}{2}}(y,t)|^2 dy\leq(\alpha-\beta)\kappa^{-1}\int_{B(x,\delta)}u^{-\alpha-1}(y,t)|\nabla u(y,t)|^2dy,$$
we can also verify \eqref{4.8} for $\delta<\min\{\frac{\delta_0}{2},\sqrt{\frac{\kappa^{(\alpha-1)/\beta}}{4C(\varepsilon)\beta\mu}}\}$ in \eqref{eq:dt-F}. Then taking  $\varepsilon=\frac12$ in \eqref{4.8},
we obtain \eqref{4.4}.
\end{proof}
Now we are in a position to prove Theorem \ref{Th:1.2}.

\noindent
{\bf Proof of Theorem \ref{Th:1.2}.} From \eqref{4.4}, we have
$$F(x,t)\leq \frac{1}{(4\pi t)^{\frac{N}{2}}}\int_{\mathbb R^N}e^{-\frac{|x-y|^2}{4t}}F_0(y)dy-\int_0^t\frac{1}{(4\pi (t-s))^{\frac{N}{2}}}\int_{\mathbb R^N}e^{-\frac{|x-y|^2}{4(t-s)}}D(y,s)dyds,$$
from which we obtain
$$\int_0^t\int_{\mathbb R^N}\frac{e^{-\frac{|x-y|^2}{4(t-s)}}}{(4\pi (t-s))^{\frac{N}{2}}}D(y,s)dyds\leq \|F_0\|_{L^\infty(\mathbb R^N)}.$$
Due to the fact that $u$ is a classical solution, we have that
$$\displaystyle\int_{\mathbb R^N}\frac{e^{-\frac{|x-y|^2}{4(t-s)}}}{(4\pi (t-s))^{\frac{N}{2}}}D(y,s)dy\in C^{2,1}(\mathbb R^N\times (0,\infty]),$$ which implies that for all $x\in\mathbb R^N$, the following limit holds:
$$\lim_{t\to\infty}\lim_{s\to t}\int_{\mathbb R^N}\frac{e^{-\frac{|x-y|^2}{4(t-s)}}}{(4\pi (t-s))^{\frac{N}{2}}}D(y,s)dy=0,$$
or equivalently,
$$
\lim_{t\to\infty}\lim_{s\to t}\int_{\mathbb R^N}\frac{e^{-\frac{|x-y|^2}{4(t-s)}}}{(4\pi (t-s))^{\frac{N}{2}}}\int_{B(y,\delta)}(\kappa^{-1}-u^\beta(z,s))^2dzdy=0.
$$
Together with the fact that the heat kernel converges to delta function as $s\rightarrow t$, we have
that for any $x\in\mathbb R^N$,
$$\lim_{t\to\infty}\int_{B(x,\delta)}(\kappa^{-1}-u^\beta(y,t))^2dy=0,$$
from which we can obtain the uniform convergence of solutions in $B(x,\delta)$, namely
$$\|u^\beta-\kappa^{-1}\|_{L^\infty(B(x,\delta))}\to0$$
as $t\to\infty$. Then for any compact set in $\mathbb R^N$, by finite covering, we obtain that $u$ converges to $\kappa^{-\frac1\beta}$ uniformly in that compact set, which means that $u$ converges locally uniformly to $\kappa^{-\frac1\beta}$ in $\mathbb R^N$ as $t\to\infty$.
Theorem \ref{Th:1.2} is proved.

\section{Numerical simulations and discussion}\label{sec:num-discussion}

In Sections \ref{sec:global-exist} and \ref{sec:hair-trigger}, we established global boundedness and the hair trigger effect of solutions to the nonlinear nonlocal reaction-diffusion initial value problem
\eqref{eq:1-1}, \eqref{eq:1-2}. The obtained results provide information about the relationship between the exponents $\alpha $ (weak Allee effect) and $\beta $ (overcrowding effect) in \eqref{eq:1-1}. The deduction of \eqref{eq:1-1} performed in the Appendix suggests that the whole dynamics is controlled by the interaction strengths $\alpha ,\beta $, the spatial dimension $N$, the kernel $J$, and the population carrying capacity encoded via nondimensionalization in the strength of the source term, thus on the constants $\mu $ and $\kappa $ below 
(for simplicity we assumed the speed of individuals to be constant).

The constant $\alpha^*$ offering an upper bound for the exponent $\alpha $ was found here to depend on $\beta $ and $N$; moreover, it is uniform with respect to the kernel $J$. By introducing the nonlocal competition term $J*u$, the  $\alpha$-interval $(1,1+\frac{2}{N})$ leading to blow-up for $\pd ut=\Delta u+u^\alpha$ is turned into an interval providing global existence for $\pd ut=\Delta u+\mu u^\alpha(1-\kappa J*u)$. Furthermore, by introducing the exponent $\beta>1$ to the nonlocal competition term, the $\alpha$-interval  $[1,1+\frac{2}{N})$ ensuring global existence is further enlarged to $\alpha\in[1,1+\frac{2\beta}{N})$, for which the upper bound is increasing with $\beta$.



Next, by numerical simulations, we also provide some clues for further investigations on the effect of the kernel $J$ on the solution behavior.
In order to study the influence of the kernel $J(x)$ on the global boundedness and the hair trigger effect, we introduce a parameter $\sigma$ and consider the following equation with the rescaled kernel $J_{\sigma}$:
\begin{align}\label{eq:simulated0}
&u_t=u_{xx}+u^{\alpha}(1-\kappa J_{\sigma}*u^\beta),
\end{align}
where $J_\sigma(x)=\frac{1}{\sigma}J(x/\sigma)$ is the $\sigma$-parametrized kernel satisfying \eqref{eq:1-3}.
Up to a rescaling, \eqref{eq:simulated0} becomes
the following $(\alpha, \beta, \mu,\kappa)$-parametrized equation
\begin{align}\label{eq:simulated}
&u_t=u_{xx}+\mu u^{\alpha}(1-\kappa J*u^\beta),
\end{align}
with $\mu=\sigma^2$.

\subsection{Simulations related to the effect of the kernel on the global boundedness}\label{sec:GB}

Following the algorithm in \cite{5} we perform numerical simulations in 1D for the initial value problem \eqref{eq:simulated} and test several combinations of $\mu, \alpha$. 
For $J$ we choose either the uniform kernel $J(x)=\frac{1}{2}\mathds 1_{[-1,1]}$ or the so-called logistic kernel $J(x)=\frac{1}{2+e^x+e^{-x}}$, see e.g. \cite{davier}. In order to handle the problem on the whole $\R$ we consider as in \cite{5} a bounded interval $(x_l,x_r)\subset \R$ and set $u\equiv 1$ on $(-\infty,x_l]$, and $u\equiv 0$ on $[x_r,\infty )$. We take the initial condition

\begin{tabular}{cc}
\parbox{10cm}{\begin{equation*}
u_0(x)=\left \{\begin{array}{cc}
1,&\text{for }x\le x_l\\
e^{-(x-x_l)^2},&\text{for }x_l<x\le 0\\
e^{-x_l^2}(1-\frac{x}{x_r}),&\text{for }0<x\le x_r\\
0,&\text{for }x>x_r.
\end{array}\right .
\end{equation*}}&\parbox{6cm}{\includegraphics[width=0.35\textwidth]{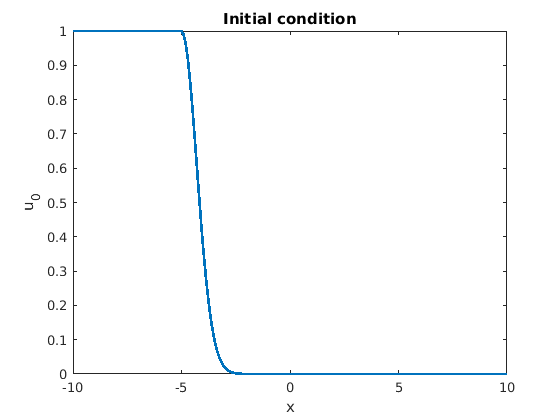}}

\end{tabular}

In our simulations $x_l=-5$, $x_r=5$, and $\beta=\kappa =1$. Figure \ref{fig:1} shows solution profiles of $u$ for $J$ being the uniform kernel with several different values of $\mu$ and $\alpha$.


\begin{figure}[h!]
	\hspace*{-2cm}	
	\begin{tabular}{ccc}
\parbox{5.2cm}{
	\begin{subfigure}[b]{.56\textwidth}
		\begin{center}
			\includegraphics[width=0.57\textwidth]{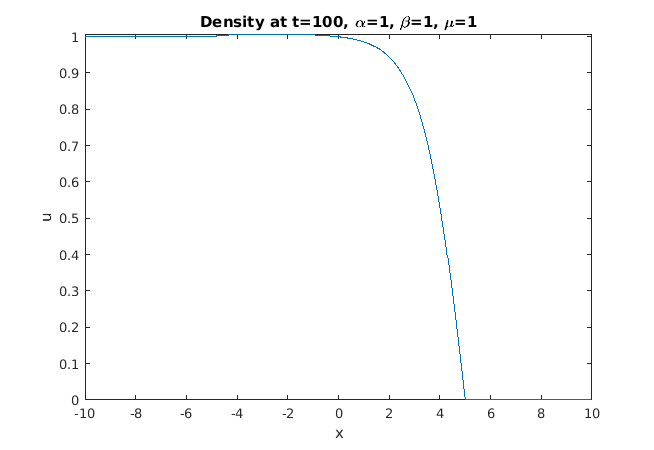}
			\caption{$\alpha =1$}\label{fig:0-1}
\end{center}\end{subfigure}}&\parbox{5.2cm}{
	\begin{subfigure}[b]{.56\textwidth}
		\begin{center}				
			\includegraphics[width=0.57\textwidth]{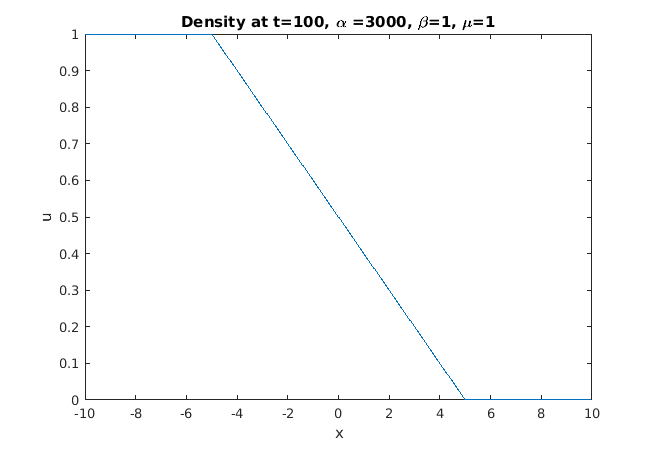}
			\caption{$\alpha =3000$}\label{fig:0-2}
\end{center}\end{subfigure}}&
\parbox{5.2cm}{
	\begin{subfigure}[b]{.57\textwidth}
		\begin{center}
			\includegraphics[width=0.56\textwidth]{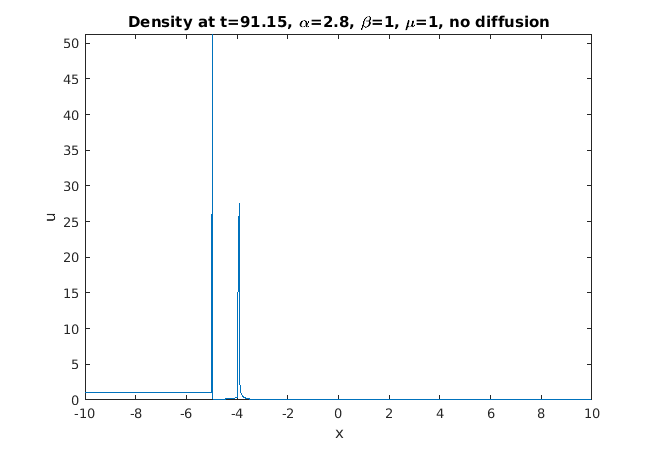}
			\caption{$\alpha =2.8$, no diffusion.}\label{fig:0-3}	
\end{center}\end{subfigure}}	\\	
		\parbox{5.2cm}{
			\begin{subfigure}[b]{.56\textwidth}
				\begin{center}
					\includegraphics[width=0.57\textwidth]{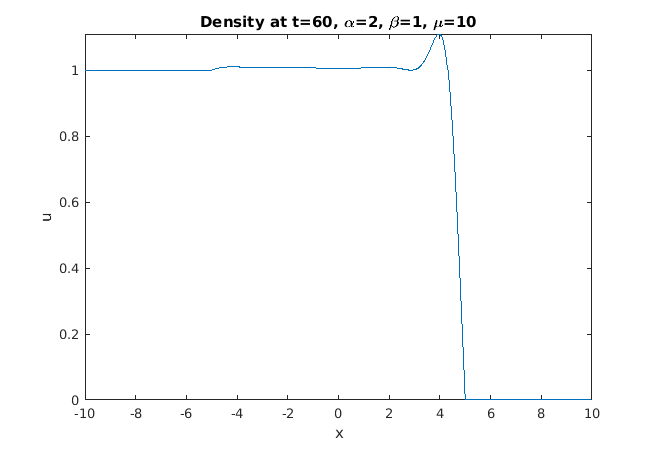}
					\caption{$\alpha =2$}\label{fig:1-1}
		\end{center}\end{subfigure}}&\parbox{5.2cm}{
			\begin{subfigure}[b]{.56\textwidth}
				\begin{center}				
					\includegraphics[width=0.57\textwidth]{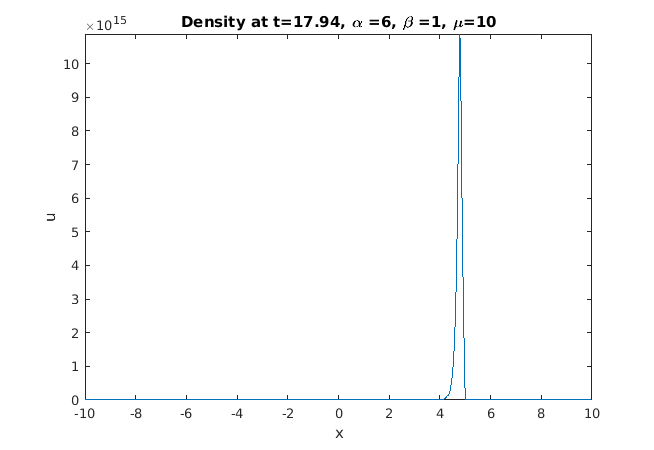}
					\caption{$\alpha =6$}\label{fig:1-2}
		\end{center}\end{subfigure}}&
		\parbox{5.2cm}{
			\begin{subfigure}[b]{.57\textwidth}
				\begin{center}
					\includegraphics[width=0.56\textwidth]{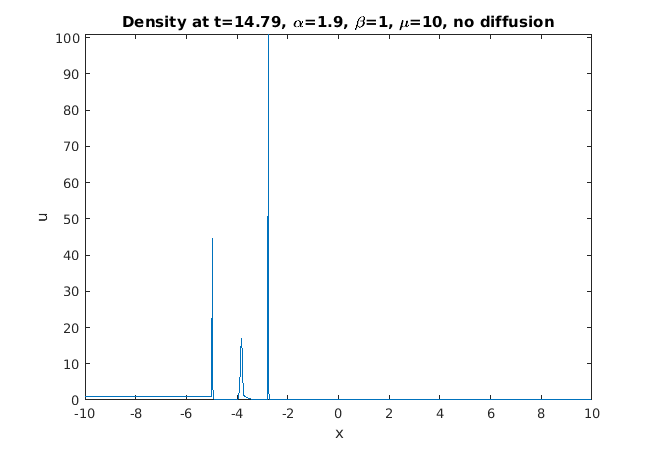}
					\caption{$\alpha =1.9$, no diffusion.}\label{fig:1-3}	
		\end{center}\end{subfigure}}	\\
		\parbox{5.2cm}{
			\begin{subfigure}[b]{.56\textwidth}
				\begin{center}
					\includegraphics[width=0.57\textwidth]{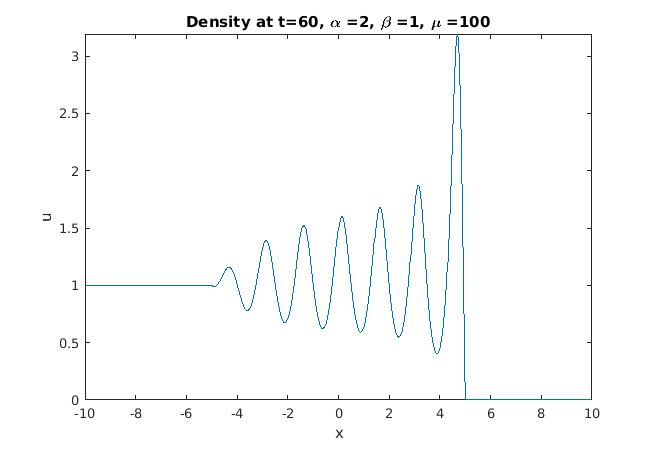}
					\caption{$\alpha =2$}\label{fig:1-4}
		\end{center}\end{subfigure}}&\parbox{5.2cm}{
			\begin{subfigure}[b]{.56\textwidth}
				\begin{center}				
					\includegraphics[width=0.57\textwidth]{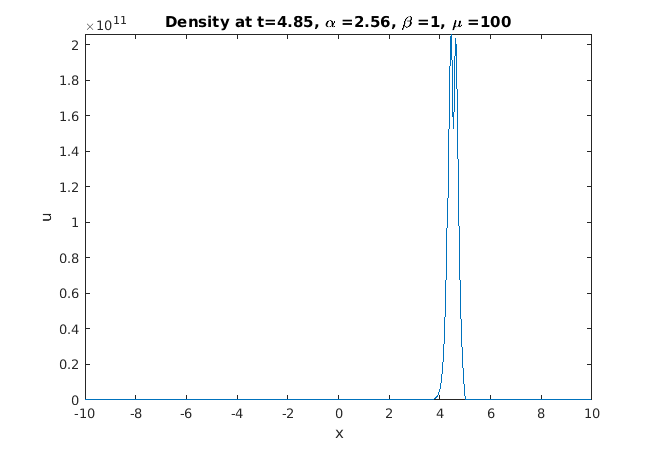}
					\caption{$\alpha =2.56$}\label{fig:1-5}
		\end{center}\end{subfigure}}&\parbox{5.2cm}{
			\begin{subfigure}[b]{.56\textwidth}
				\begin{center}				
					\includegraphics[width=0.57\textwidth]{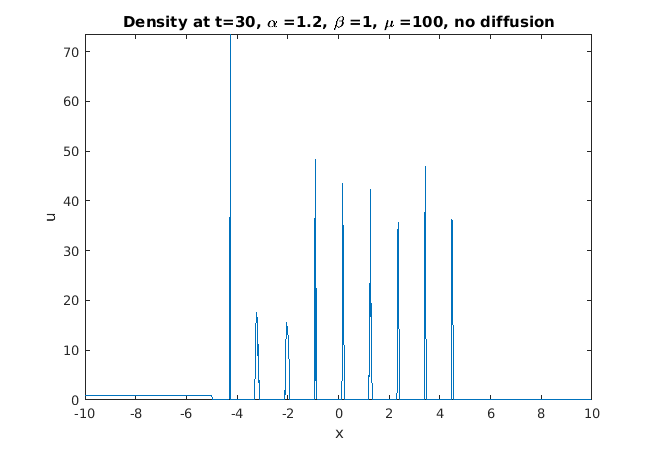}
					\caption{$\alpha =1.2$, no diffusion}\label{fig:1-6}
		\end{center}\end{subfigure}}			
	\end{tabular}
	\begin{center}
		\caption{Simulation results for \eqref{eq:simulated} with $\beta =\kappa=1$, $\mu=1$ (upper row), $\mu =10$ (middle row), $\mu =100$ (lower row), $J$ uniform kernel, and different values of $\alpha$. Subfigures \ref{fig:0-3}, \ref{fig:1-2}, \ref{fig:1-3}, \ref{fig:1-5} and \ref{fig:1-6} show the solution at the time moments just before it blows up.}\label{fig:1}
	\end{center}
\end{figure}


Our numerical experiments show that the solution stays bounded for any values of $\alpha $ when $\mu $ is rather small, e.g. for $\mu =1$. The corresponding solution profiles look like those in Subfigures \ref{fig:0-1}, \ref{fig:0-2}; the only change noticed for different concrete values of $\alpha$ is in the curve connecting the levels $u=1$ and $u=0$ at $x=x_l$ and $x=x_r$, respectively: For small $\alpha $\footnote{less than approximately $\alpha =15$} this is a  shoulder curve which can slightly and transiently exceed the level $u=1$ (see Subfigure \ref{fig:0-1}), while for $\alpha $ large it becomes a straight line, as in Subfigure \ref{fig:0-2}. For $\mu =10$ the solution is still bounded, even if $\alpha $ exceeds the upper bound (here $\alpha ^*=2$) obtained in our analysis. It first became unbounded for $\alpha =6$. We also explored (still for $\mu =10$) the global boundedness of the solution in the case with no diffusion, for which there are no theoretical results: In this situation the solution was found to blow up already for $\alpha =1.9$, see Subfigure \ref{fig:1-3}. This indicates that neglecting diffusion leads to insufficient dampening of the growth, which triggers unboundedness even for smaller values of $\alpha$. Moreover, in this case the combination $\mu =1$ and $\alpha \ge 2.8$ leads to blow-up as well, see Subfigure \ref{fig:0-3}.


Simulation results for $\mu=100$ are shown in Subfigures \ref{fig:1-4}-\ref{fig:1-6}. This increase of the interaction rate reduces the range of $\alpha $ in which there is no blow-up occurring. Our tests showed that the latter already happens for $\alpha =2.56$, while the solution stays bounded for $\alpha$ below that value, although it can exhibit a highly oscillatory behavior, occasionally with very large and locally concentrated aggregates, as shown in Subfigures \ref{fig:1-4}, \ref{fig:1-6}. Such peaks of the solution can suddenly emerge, grow very fast, and then stabilize or drop back to small values. Neglecting diffusion has for $\mu=100$ the same effect as above for $\mu =10$: It leads to blow-up of the solution, and this already for even smaller values of $\alpha $.

\begin{figure}[h!]
	\hspace*{-2cm}	
	\begin{tabular}{ccc}
		\parbox{5.2cm}{
			\begin{subfigure}[b]{.56\textwidth}
				\begin{center}
					\includegraphics[width=0.57\textwidth]{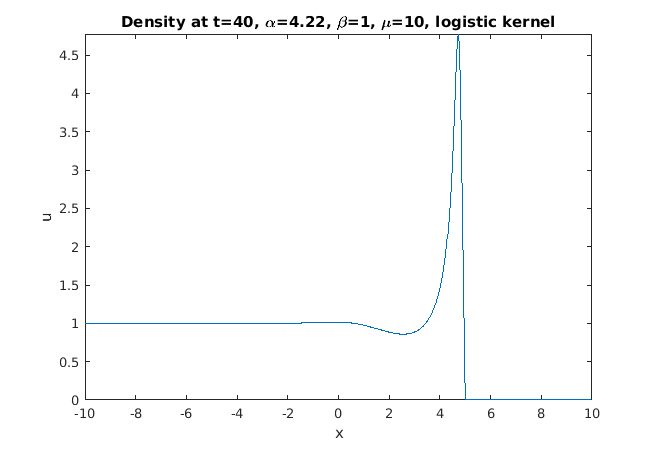}
					\caption{$\alpha =4.22,\ \mu =10$}\label{fig:2-1}
		\end{center}\end{subfigure}}&\parbox{5.2cm}{
			\begin{subfigure}[b]{.56\textwidth}
				\begin{center}				
					\includegraphics[width=0.57\textwidth]{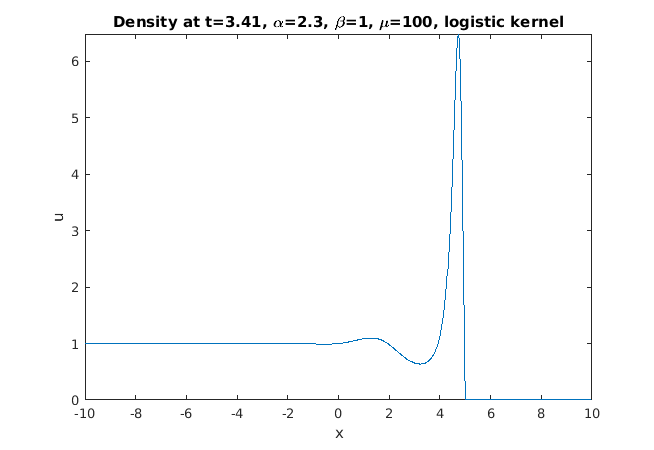}
					\caption{$\alpha =2.3,\ \mu=100$}\label{fig:2-2}
		\end{center}\end{subfigure}}&
		\parbox{5.2cm}{
			\begin{subfigure}[b]{.57\textwidth}
				\begin{center}
					\includegraphics[width=0.56\textwidth]{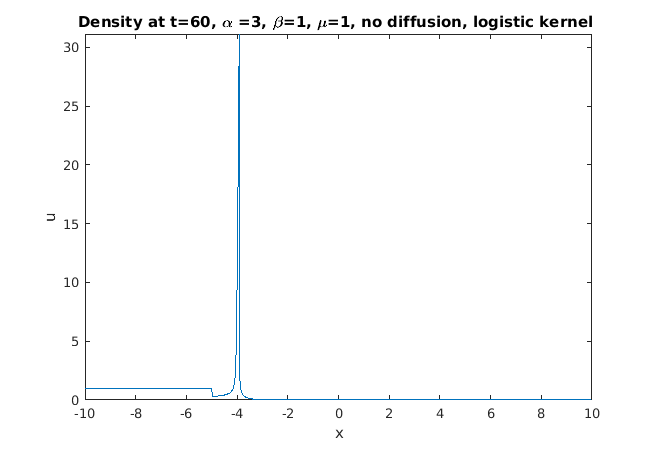}
					\caption{$\alpha =3,\ \mu =1$, no diffusion.}\label{fig:2-3}	
		\end{center}\end{subfigure}}					
	\end{tabular}
	\begin{center}
		\caption{Simulation results for \eqref{eq:simulated} with $J$ logistic kernel, $\beta=\kappa=1$.}\label{fig:2}
	\end{center}
\end{figure}

Figure \ref{fig:2} illustrates solution profiles of $u$ for $J$ being the logistic kernel, $\beta=\kappa=1$, and different $\mu $ and $\alpha $ combinations. The solution behavior is similar to that for $J$ being uniform, but now blow-up occurs (for the same values of $\mu$) at lower values of $\alpha$: For $\mu =10$ the solution explodes for $\alpha \ge 4.23$, and in case $\mu =100$ for $\alpha \ge 2.3$. Subfigures \ref{fig:2-1} and  \ref{fig:2-2} show solution profiles for $\mu =10$ and $\alpha =4.22$ (no blow-up) and for $\mu =100$ and $\alpha=2.3$ (at the time just before blow-up), respectively. The shapes of the solutions are quite alike, just with higher maxima for $\mu$ increasing. In all simulations with this type of kernel there are far less oscillations and peaks, which get damped rather fast, only one dominant aggregate remaining.
Interestingly, for $D=0$ and $\mu =1$ the first exponent for which blow-up occurs is $\alpha =3.1$, which is larger than in the case where $J$ was uniform, compare Subfigure \ref{fig:0-3}. 


\subsection{Simulations for the influence of the kernel on the hair trigger effect and pattern formation}\label{sec:HTE}

In order to test the hair trigger effect and to get some insight into the qualitative behavior of the solution we also performed numerical simulations for different values of $\mu$ and two different kernels.

\noindent
We start with the case $\kappa=\beta=1$ and several combinations of the parameters $\alpha$ and $\mu$. As before, $J$ is taken to be the uniform or the logistic kernel.
 The results are shown in Figures \ref{fig:3} and \ref{fig:4} for the uniform kernel and for the logistic kernel, respectively. From Subfigure \ref{fig:3-1} we notice that for $\mu=1$ and $\alpha <\alpha^*$, the solution converges locally uniformly to $1$, which is the hair trigger effect. Subfigures  \ref{fig:3-2}, \ref{fig:3-3} show that for $\mu=50$ and $150$, respectively, the solution  forms different patterns, larger $\mu$ values leading to more oscillatory patterns. These facts suggest that the smallness assumption on $\mu$ for hair trigger effect is necessary. A similar behavior is observed for $\alpha$ coinciding with or being slightly beyond $\alpha^*$, while the solution explodes for $\alpha \ge \hat \alpha >\alpha^*=2$, the critical value $\hat \alpha $ depending as before on the choice of $\mu $, compare Subfigures \ref{fig:3-4} and \ref{fig:3-5}.
Allowing for more frequent oscillations in the initial condition leads to the same behavior, however with singularities occurring at later times. Simulations with the same initial condition, but with $J$ being the logistic kernel are illustrated in Figure \ref{fig:4}. The hair trigger effect for  small $\mu$ values is similar, however the solution exhibits less oscillations, which means that the shape of patterns is strongly influenced by the choice of the kernel. Furthermore, for the logistic kernel the solution blows up earlier, and for smaller values of $\alpha $: Subfigure \ref{fig:4-4} shows the solution profile for $\alpha =3.8$ shortly before blow-up; compare with Subfigure \ref{fig:3-5}.


\begin{figure}[h!]
	\hspace*{-2cm}	
	\begin{tabular}{cc}
		\parbox{7.7cm}{
			\begin{subfigure}[b]{.67\textwidth}
			\begin{center}
				\includegraphics[width=0.77\textwidth]{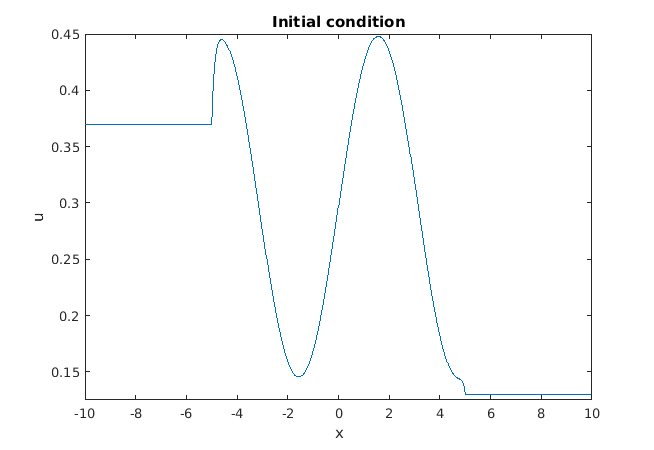}	
				\caption{Initial condition}\label{fig:3-0}
					\end{center}\end{subfigure}}&\parbox{7.7cm}{\begin{subfigure}[b]{.67\textwidth}
				\begin{center}
				\includegraphics[width=0.77\textwidth]{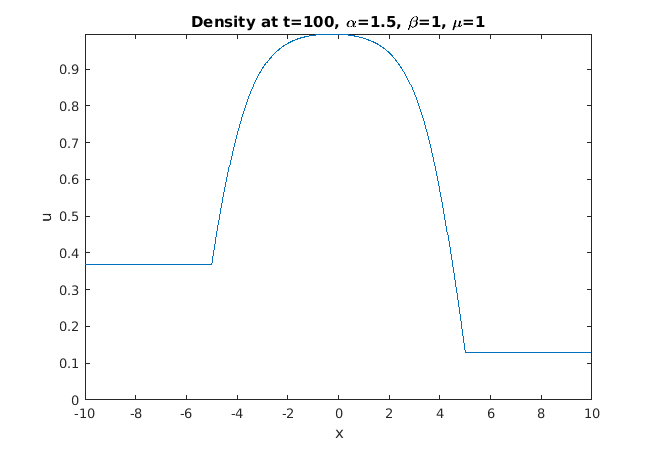}
				\caption{$\alpha =1.5,\ \mu =1$}\label{fig:3-1}	
			\end{center}\end{subfigure}}\\
					\parbox{7.7cm}{
				\begin{subfigure}[b]{.67\textwidth}
				\begin{center}
					\includegraphics[width=0.77\textwidth]{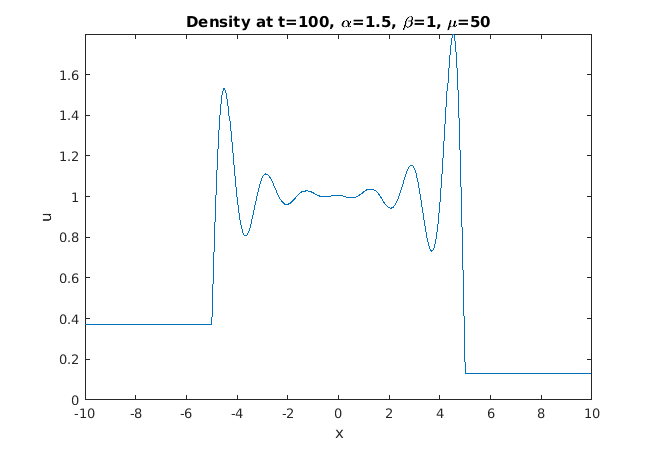}
					\caption{$\alpha =1.5,\ \mu =50$}\label{fig:3-2}	
		\end{center}\end{subfigure}}&
\parbox{7.7cm}{
	\begin{subfigure}[b]{.67\textwidth}
		\begin{center}
			\includegraphics[width=0.77\textwidth]{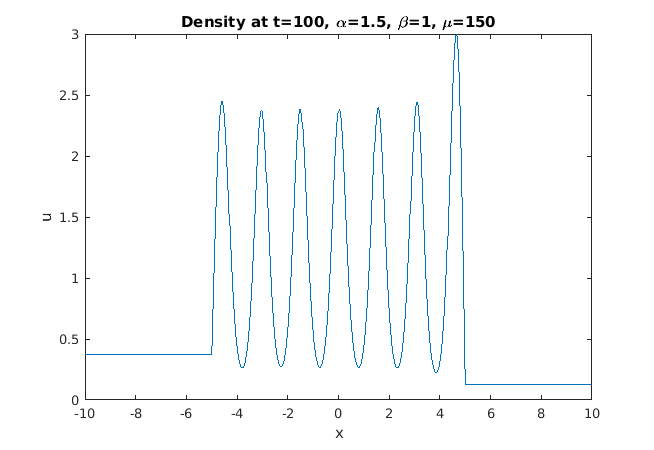}
			\caption{$\alpha =1.5,\ \mu =150$}\label{fig:3-3}
\end{center}\end{subfigure}}\\
\parbox{7.77cm}{
	\begin{subfigure}[b]{.67\textwidth}
		\begin{center}				
			\includegraphics[width=0.77\textwidth]{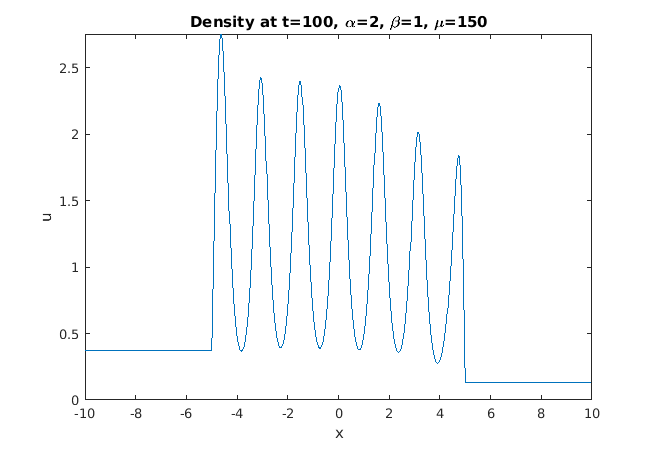}
			\caption{$\alpha =2,\ \mu=150$}\label{fig:3-4}
\end{center}\end{subfigure}}&\parbox{7.7cm}{\begin{subfigure}[b]{.67\textwidth}
\begin{center}
	\includegraphics[width=0.77\textwidth]{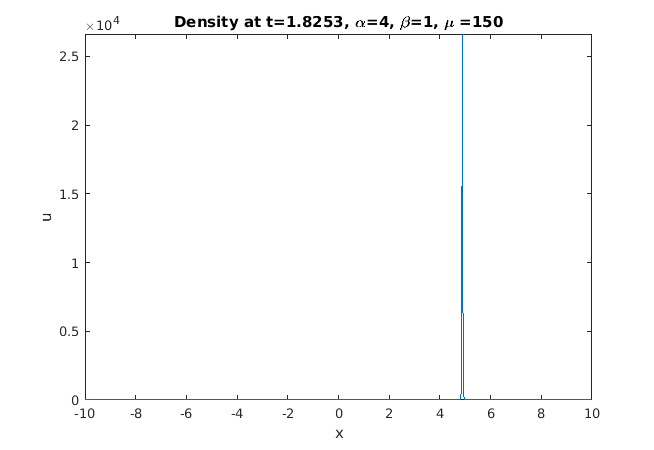}
	\caption{$\alpha =4,\ \mu =150$}\label{fig:3-5}
\end{center}\end{subfigure}}
		\end{tabular}
\begin{center}
\caption{Initial condition and simulation results for \eqref{eq:simulated} with $J$ uniform kernel, $\beta=\kappa=1$.}\label{fig:3}
\end{center}
\end{figure}

\begin{figure}[h!]
	\hspace*{-2cm}	
	\begin{tabular}{cc}
		\parbox{7.7cm}{
			\begin{subfigure}[b]{.67\textwidth}
				\begin{center}
					\includegraphics[width=0.77\textwidth]{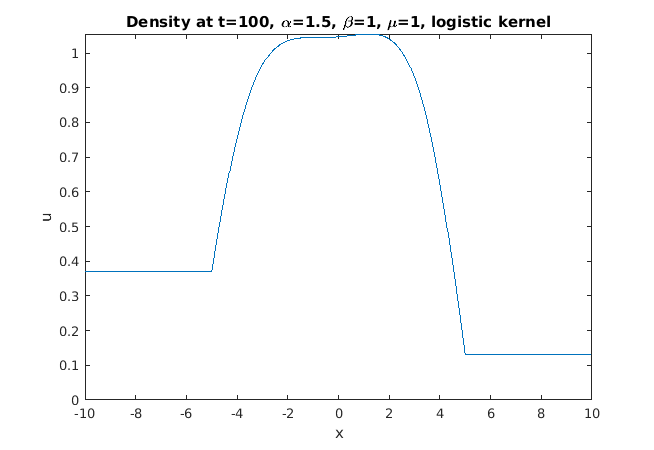}	
					\caption{$\alpha =1.5,\ \mu =1$}\label{fig:4-1}
		\end{center}\end{subfigure}}&\parbox{7.7cm}{\begin{subfigure}[b]{.67\textwidth}
				\begin{center}
					\includegraphics[width=0.77\textwidth]{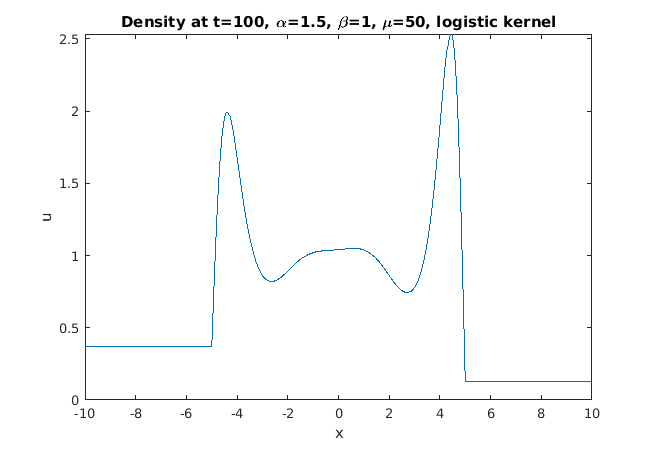}
					\caption{$\alpha =1.5,\ \mu =50$}\label{fig:4-2}	
		\end{center}\end{subfigure}}\\
		\parbox{7.7cm}{
			\begin{subfigure}[b]{.67\textwidth}
				\begin{center}
					\includegraphics[width=0.77\textwidth]{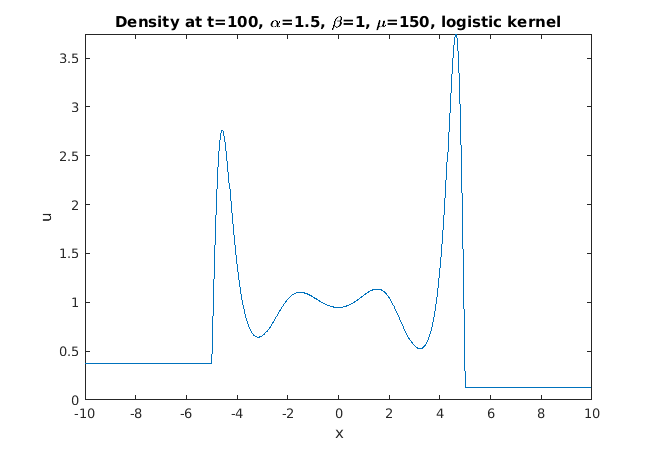}
					\caption{$\alpha =1.5,\ \mu =150$}\label{fig:4-3}	
		\end{center}\end{subfigure}}&
		\parbox{7.7cm}{
			\begin{subfigure}[b]{.67\textwidth}
				\begin{center}
					\includegraphics[width=0.77\textwidth]{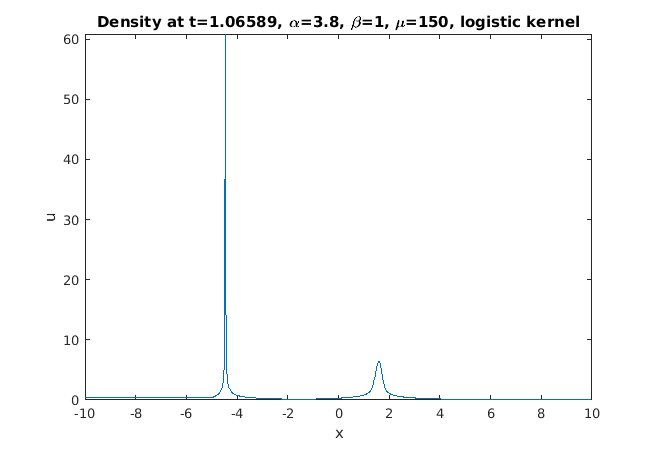}
					\caption{$\alpha =3.8,\ \mu =150$}\label{fig:4-4}
		\end{center}\end{subfigure}}
	\end{tabular}
	\begin{center}
		\caption{Simulation results for \eqref{eq:simulated} with $J$ logistic kernel, $\beta=\kappa=1$.}\label{fig:4}
	\end{center}
\end{figure}

We also tested (still with the initial condition in Subfigure  \ref{fig:3-0}) the situation with small parameters $\beta $ and $\kappa $, which means huge values of the constant steady-state $\kappa ^{-1/\beta}$. According to 
	Theorem \ref{Th:1.1} it should be possible for the solution to stay bounded, although the upper bound might be very large. This is indeed the case; Figure \ref{fig:5} shows simulations done for $\beta =0.1$ and $\kappa =0.2$ when $\alpha =1.1$ (thus $\alpha =\alpha ^*$) and $\mu =150$. The solution oscillates shortly, but quickly stabilizes at $\kappa ^{-1/\beta}\simeq 0.98e+07$. This situation is, however, more prone to blow-up than that for larger $\beta , \kappa $: this happens already for $\alpha =1.173$, thus slightly exceeding $\alpha ^*$, in contrast to the results in Figure \ref{fig:3}, where blow-up first occurred for $\alpha \simeq 4$.
	
	When $\kappa$ becomes very small (thus for a very weak depletion due to intrapopulation concurrence) the solution can infer very large values. Simulations at three different times are shown in Figure \ref{fig:6} for $\mu=48$; the solution exhibits oscillations of huge amplitude around its expected asymptotic limit, ut stays bounded.    
	
	 The choice $\beta =0.01$, $\kappa =0.1$ leads to an enormous upper bound and a longer time needed for the solution to stabilize, but no explosion, as long as $\alpha <\alpha^*$ and for an arbitrary, fixed $\mu  $. 
	 Hence, there seems to exist a fine and rather complex tuning between the parameters of the problem (whereby for the same $\mu$ too small values of $\kappa$ are more favorable to producing large bounds than too small $\beta$'s) and the choice of $J$ may also play a role, although presumably a less prominent one.\\

\begin{figure}[h!]
	\hspace*{-2cm}	
	\begin{tabular}{ccc}
		\parbox{5.2cm}{
			\begin{subfigure}[b]{.56\textwidth}
				\begin{center}
					\includegraphics[width=0.57\textwidth]{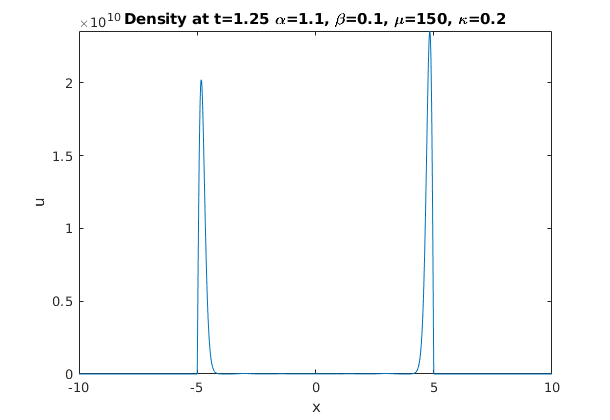}
					\caption{$t=1.25$}\label{fig:5-1}
		\end{center}\end{subfigure}}&\parbox{5.2cm}{
			\begin{subfigure}[b]{.56\textwidth}
				\begin{center}
					\includegraphics[width=0.57\textwidth]{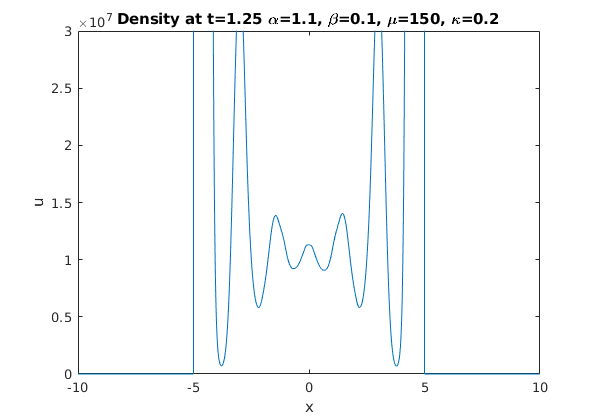}
					\caption{$t=1.25$, zoomed}\label{fig:5-2}	
		\end{center}\end{subfigure}}
		&\parbox{5.2cm}{\begin{subfigure}[b]{.56\textwidth}
				\begin{center}
					\includegraphics[width=0.57\textwidth]{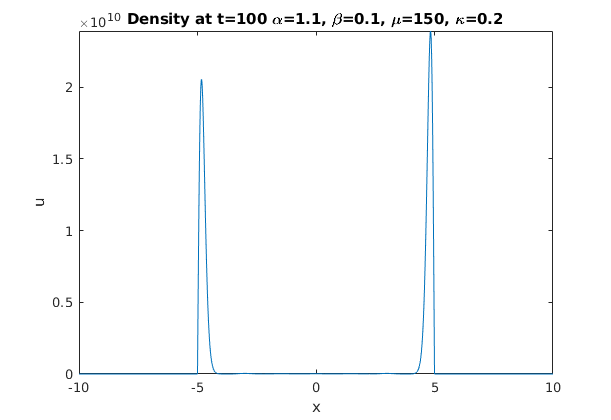}
					\caption{$t=100$}\label{fig:5-3}	
		\end{center}\end{subfigure}}
	\end{tabular}
	\begin{center}
		\caption{Simulation results for \eqref{eq:simulated} with $J$ uniform, $\alpha =1.1,\ \beta=0.1,\ \kappa=0.2,\ \mu =150$ }\label{fig:5}
	\end{center}
\end{figure}

\begin{figure}[h!]
	\hspace*{-2cm}	
	\begin{tabular}{ccc}
		\parbox{5.2cm}{
			\begin{subfigure}[b]{.56\textwidth}
				\begin{center}
					\includegraphics[width=0.57\textwidth]{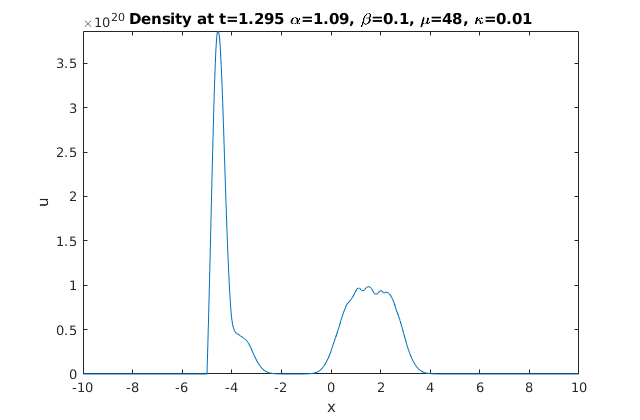}
					\caption{$t=1.295$}\label{fig:5-1}
		\end{center}\end{subfigure}}&\parbox{5.2cm}{
			\begin{subfigure}[b]{.56\textwidth}
				\begin{center}
					\includegraphics[width=0.57\textwidth]{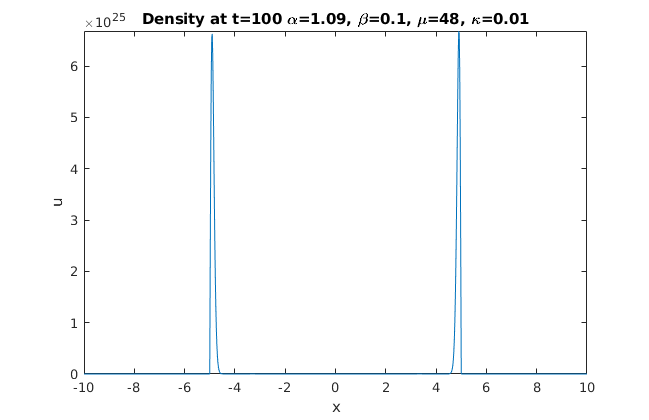}
					\caption{$t=100$}\label{fig:5-2}	
		\end{center}\end{subfigure}}
		&\parbox{5.2cm}{\begin{subfigure}[b]{.56\textwidth}
				\begin{center}
					\includegraphics[width=0.57\textwidth]{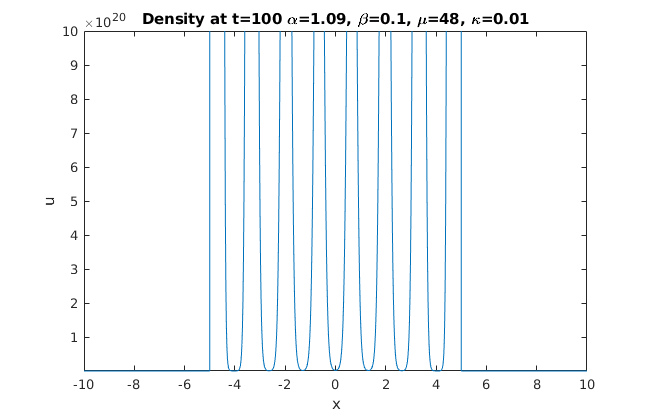}
					\caption{$t=100$, zoomed}\label{fig:5-3}	
		\end{center}\end{subfigure}}
	\end{tabular}
	\begin{center}
		\caption{Simulation results for \eqref{eq:simulated} with $J$ uniform, $\alpha =1.09,\ \beta=0.1,\ \kappa=0.01,\ \mu =48$ }\label{fig:6}
	\end{center}
\end{figure}

\subsection{Discussion}
The simulation-based observations in \ref{sec:GB} and \ref{sec:HTE} suggest that the solution behavior w.r.t. to global boundedness and patterning is influenced not only by the values of $\alpha$, $\beta$, but also by $\mu$, $\kappa$ and the shape of the convolution kernel $J$. Moreover, it seems that the ($\beta$-dependent) bounds established for $\alpha$ in this paper could be non-sharp. 
It would be interesting to investigate the conditions under which the solution ceases to remain bounded. Concerning the form of the convolution kernel, we expect that there is some $\hat{\alpha}(J,\mu)\geq\alpha^*$, such that for any initial data, the solution still exists and stays bounded for $\alpha ^*<\alpha <\hat{\alpha}(J,\mu)$, whereas for $\alpha \ge \hat{\alpha}(J,\mu)$, there exists initial data such that blow-up occurs.
Whether and when this model can exhibit pattern formation remains unsolved.

We handled in this paper a PDE describing the dynamics of a single species under linear diffusion and nonlocal intrapopulation interactions. In many applications, however, it turns out that other types of diffusion might be more appropriate to characterize the behavior of a certain population. For instance, the so-called myopic diffusion $\nabla \nabla : (\mathbb D(x)u)=\nabla \cdot (\nabla \cdot\mathbb D(x) u+\mathbb D(x)\nabla u)$ has been obtained in connection with the anisotropic spread of brain tumor cells in the surrounding tissue, see e.g. \cite{EHKS,HP} and references therein. This kind of diffusion corresponds to the cells perceiving their surroundings right where they are; the respective PDE is most often obtained either from master equations written for position jumps between the sites of a lattice (where the transition probabilities depend on the information available at the current cell locations) or from velocity jump processes on the mesoscale, with a subsequent, adequate upscaling (usually of the parabolic type). 
Further types of diffusion lead to quasilinear equations where the diffusion coefficients depend on the solution in a more or less complicated way. The analysis of such equations with nonlocality is itself a nontrivial problem. Yet other interesting issues relate to nonlocal interactions between at least two different populations or between a population of individuals performing a certain type of tactic motion towards or away from some diffusing or nondiffusing signal. For a review on related nonlocal models we refer to \cite{CPSZ} and for more comprehensive reviews in a broader context to e.g., \cite{eftimie,kavall-suzuki}.

\section*{Appendix: deduction of an equation of type \eqref{eq:1-1} from a mesoscopic formulation}

We start with the kinetic transport equation
\begin{equation}\label{eq:KTE1}
p_t+v\cdot \nabla _xp=\lambda \mathcal L[p]+ \tilde \mu \mathcal I[p,p],
\end{equation}
where $p(x,t,v)$ represents the distribution function of individuals being at time $t$ in position $x$ and having velocity $v\in V$. The velocity space $V$ is assumed to be bounded, e.g. of the form $V=[s_1,s_2]\times \mathbb S^{N-1}$, where $s_1, s_2$ denote the minimal, respectively maximal speed of an individual. The right hand side operators describe reorientations of individuals and growth/decay of $p$ due to proliferative/competitive interactions. The coefficients $\lambda , \tilde \mu >0$ represent the turning frequency and the interaction rate, respectively, and are assumed here to be constants. By an appropriate rescaling \footnote{We assume that the turning time, i.e. $\frac{1}{\lambda}$ is $\eps $-small when compared to the characteristic time $\tau $ of the mesoscopic dynamics described by \eqref{eq:KTE1}.
	Moreover, $\tilde \mu $ is assumed to be much smaller than $\lambda $:
	the individuals have a high preference of changing direction rather that interacting and crowding.} 
we get $\lambda =\frac{1}{\eps}$, $\tilde \mu =\eps$, and the above equation \eqref{eq:KTE1} becomes
\begin{equation}\label{eq:KTE2}
\eps p_t+v\cdot \nabla _xp=\frac{1}{\eps} \mathcal L[p]+ \eps  \mathcal I[p,p].
\end{equation}

We define the integral operators as follows:
\begin{align}
\mathcal L[p](x,t,v)&=\int _V\Big (T(v,v')p(x,t,v')-T(v',v)p(x,t,v)\Big ) dv',\\
\mathcal I[p,p](x,t,v)&=\frac{p^{\alpha }(x,t,v)}{\int _VM^{\alpha }(v)dv}-\frac{\tilde \kappa }{\int _VM^{\alpha +\beta }(v)dv}p^{\alpha }(x,t,v)\int _\Om J(x,x')p^{\beta }(x',t,v)dx'.
\end{align}

Thereby, $\alpha ,\beta \ge 1$ and $\tilde \kappa >0$ are constants, $J(x,x')$ is a function weighting the interactions between (groups of) individuals having the same velocity regime within $\Om \subseteq \R^N$ \footnote{correspondingly normalized if $\Om$ is bounded}. We can think e.g., of collectives of cells which are attached to their neighbors and move together in (roughly) the same direction and having the same (average) speed. Two such collectives align their motion (usually the smaller group to the larger one) and correspondingly adapt their (mean) direction. They have to be sufficiently large to ensure proliferation (cells 'feel well' among a moderate bunch of their own kind), but not too large to compete excessively for space or other resources. 
	Our assumption of the cells sharing the same velocity regime is a simplification, in order not to complicate too much the exposition. We could also think of allowing different velocities, which would require introducing yet another kernel, for transitions from one velocity regime to the other. A convenient choice of such kernel will lead to the same result as in the present setting.

We assume that
$J$ depends on the distance between the interacting (clusters of) individuals and take here $J(x,x')=J(x-x')$, also requiring $J$ to satisfy conditions \eqref{eq:1-3}. Further, $T(v,v')\ge 0$ is a turning kernel giving the likelihood of an individual having velocity $v'$ to assume the new velocity $v$. The operator $\mathcal L$ and its turning kernel are supposed to satisfy the following

\begin{assumption}\quad
\begin{itemize}
\item $\int _VT(v,v')dv=1$;
\item  $\int _V\mathcal L[\phi ]dv=0$, for all $\phi $;
\item There exists a bounded velocity distribution $M(v)>0$, not depending on $t, x$,  such that the detailed balance condition $T(v,v')M(v')=T(v',v)M(v)$
holds and \\$\int _VM(v)dv=1,\ \int _VvM(v)dv=0$.
\item There exist $c,C>0$ constants such that $cM(v)\le T(v,v')\le CM(v)$, for all $v, v'\in V$, $x\in \Om $, and $t>0$.
\end{itemize}
\end{assumption}
The following result holds:
\begin{lemma}(see e.g., \cite{BBNS12})
Under the above assumptions the operator $\mathcal L$ has the properties:
\begin{itemize}
	\item $\mathcal L$ is self-adjoint in the weighted space $L^2(V,\frac{dv}{M(v)})$;
	\item For $\psi \in L^2$ there is a unique $\phi \in L^2(V,\frac{dv}{M(v)})$ such that $\mathcal L[\phi ]=\psi $, which satisfies
	$$\int _V\phi (v)dv=0\qquad \text {iff}\qquad \int _V\psi (v)dv=0;$$
	\item $\Ker \mathcal L=<M(v)>$, the vector space spanned by $M(v)$;
	\item There exists a unique function $\theta (v)$ satisfying $\mathcal L[\theta (v)]=vM(v)$.
\end{itemize}
\end{lemma}
Specifically, we consider $T(v,v')=M(v)$, which obviously has the required properties. This gives $\theta (v)=-vM(v)$.

Let $u(x,t)=\int _Vp(x,t,v)dv$, with $p$ being a solution of \eqref{eq:KTE2}. We decompose $p(x,t,v)=M(v)u(x,t)+\eps g(x,t,v)$, which gives $\int _Vg(x,t,v)dv=0$ and
\begin{equation}\label{eq:afterepsdiv}
\partial_t(M(v)u)+\eps \partial_tg+\frac{1}{\eps }vM(v)\cdot \nabla _xu+v\cdot \nabla _xg=\frac{1}{\eps }\mathcal L[g]+\mathcal I[p,p].
\end{equation}

Integrating \eqref{eq:afterepsdiv} with respect to $v$ yields
\begin{equation}\label{eq:4.6}
u_t+\nabla _x\cdot \int _Vvg(x,t,v)dv=\int _V\mathcal I[p,p]dv.
\end{equation}

Observe that $\mathcal I[M(v)u+\eps g,M(v)u+\eps g]=\mathcal I[M(v)u,M(v)u]+\mathcal O(\eps)$.

Then considering as in \cite{BelloBello} the orthogonal projection operator onto $\Ker \mathcal L$ we have
\begin{align*}
&P_M(h)(v)=M(v)\int _Vh(v)dv,\qquad h\in L^2(V,\frac{dv}{M(v)})\\
&(I-P_M)(M(v)u)=P_M(g)=0\\
&(I-P_M)(vM(v)\cdot \nabla _xu)=vM(v)\cdot \nabla _xu.
\end{align*}

Apply $I-P_M$ to \eqref{eq:afterepsdiv} to obtain
\begin{equation}
\eps \partial_tg+\frac{1}{\eps }vM(v)\cdot \nabla _xu+(I-P_M)(v\cdot \nabla _xg)=\frac{1}{\eps }\mathcal L[g]+(I-P_M)(\mathcal I[p,p]),
\end{equation}
from which
\begin{equation}
g=\mathcal{L}^{-1}(vM(v)\cdot \nabla _xu)+\mathcal O(\eps).
\end{equation}

Plugging this into \eqref{eq:4.6} leads to
\begin{equation}\label{eq:4.9}
u_t+\int_Vv\cdot \nabla _x(\mathcal{L}^{-1}(vM(v)\cdot \nabla _xu))dv=\int _V\mathcal I[M(v)u,M(v)u]dv+\mathcal O(\eps).
\end{equation}

Now observe that
$$\int_Vv\cdot \nabla _x(\mathcal{L}^{-1}(vM(v)\cdot \nabla _xu))dv=\nabla_x\cdot \left (\int _Vv\otimes \theta (v)dv\cdot \nabla _xu\right),$$
therefore from \eqref{eq:4.9} we formally obtain in the limit $\eps \to 0$ the macroscopic nonlocal PDE
\begin{align*}
u_t-\D\Delta u=u^{\alpha }\left (1-\tilde \kappa J*u^\beta \right ),
\end{align*}
where $\D=\int _Vv\otimes v\ M(v)dv$. In particular, if we consider the uniform velocity distribution  $M(v)=\frac{1}{|V|}$ and assume that the individuals can have different orientations, but all preserve the same constant speed $s$, so that  $V=s\mathbb S^{N-1}$, then we have $|V|=\omega _0s^{N-1}$, with $\omega _0=|\mathbb S^{N-1}|$, leading to $\D=s^2/N$. A nondimensionalization leads to the PDE having the form \eqref{eq:1-1}.
Theorem \ref{Th:1.1} is proved.


\section*{Acknowledgement} The authors thank Aydar Uatay (TU Kaiserslautern) for helping with the implementation of the numerical method. This work is supported by the German Research Foundation DFG, grant CH 955/3-1, and the Research Grant Funds of Minzu University of China.

\end{document}